\documentclass{amsart}
\usepackage{amsmath}
\usepackage{amsthm}
\usepackage{amssymb}
\usepackage{amsfonts,amssymb,verbatim}
\usepackage{epsfig}
\usepackage{graphics}

\def\cN{\mathcal{N}}
\def\oN{\overline{N}}
\def\oD{\overline{D}}
\def\oS{\overline{S}}
\def\oP{\overline{P}}
\def\oZ{\overline{Z}}

\newcommand{\SPP}{\vspace{0.05cm}\par}

\def\bS{\mathbb{S}}

\def\bSb{\overline{\mathbb{S}}}

\input colordvi

\newtheorem{thm}{Theorem}
\newtheorem{rk}[thm]{Remark}

\newtheorem{lem}[thm]{Lemma}

\newtheorem{definition}[thm]{Definition}

\def\tcA{\widetilde{\cA}}
\def\tcB{\widetilde{\cB}}

\def\cCp{\cC^{\bullet}}
\def\cGp{\cG^{\bullet}}
\def\cZ{\mathcal{Z}}
\def\Forbid{\mathrm{Forbid}}
\def\tcG{\widetilde{\cG}}
\def\tcC{\widetilde{\cC}}
\def\tcCp{\widetilde{\cCp}}
\def\tcB{\widetilde{\cB}}
\def\expb{\overline{\exp}}
\newcommand{\Zb}[1]{Z_{\mathcal{B}'_{#1}}}
\newcommand{\tcsef}[1]{\widetilde{C}'_{#1}(z, \mathbf{v})}
\newcommand{\diff}[1]{\frac{\partial}{\partial {#1}}}
\newcommand{\cacher}[1]{}

\def\ds{\displaystyle}

\def\rm{\mathrm}

\def\Set{\mathrm{Set}}

\def\bi{\begin{itemize}}
\def\ei{\end{itemize}}
\def\be{\begin{enumerate}}
\def\ee{\end{enumerate}}
\def\btheo{\begin{thm}}
\def\etheo{\end{thm}}
\def\blem{\begin{lem}}
\def\elem{\end{lem}}
\def\beq{\begin{equation}}
\def\eeq{\end{equation}}

\newcommand{\mc}[1]{\ensuremath{\mathcal{#1}}}

\def\cA{\mathcal{A}}
\def\cC{\mathcal{C}}
\def\cD{\mathcal{D}}
\def\cB{\mathcal{B}}

\def\cG{\mathcal{G}}

\def\cS{\mathcal{S}}

\def\cM{\mathcal{M}}
\def\cP{\mathcal{P}}

\def\vy{\mathbf{y}}
\def\vv{\mathbf{v}}
\def\vvo{\mathbf{v}_0}
\def\vF{\mathbf{F}}

\def\rm{\mathrm}

\def\JacF{\mathbf{Jac}_\textbf{F}}

\begin{document}
\title[Asymptotic study of subcritical graph classes]{Asymptotic study of subcritical graph classes}
\author[Drmota, Fusy, Kang, Kraus and  Ru\'e]{Michael Drmota$^{*}$, \'{E}ric Fusy$^{\dagger}$, Mihyun Kang$^{\ddagger}$, Veronika Kraus$^{*}$ \and Juanjo Ru\'e$^{\dagger}$}
\thanks{$^{*}$Institut f\"ur Diskrete Mathematik und Geometrie, Technische Universit\"at Wien,  michael.drmota@tuwien.ac.at, vkraus@dmg.tuwien.ac.at. Supported by the Austrian Science Foundation FWF, NFN-Project 9604.\\
$^{\dagger}$LIX, \'Ecole Polytechnique, Palaiseau, fusy@lix.polytechnique.fr, rue1982@lix.polytechnique.fr. Supported by the European Research Council under
the European Community's 7th Framework Programme, ERC grant agreement no 208471 - ExploreMaps project.\\
$^{\ddagger}$Institut f\"ur Mathematik, Technische Universit\"at Berlin, kang@math.tu-berlin.de. Supported by the DFG Heisenberg Programme.}

\begin{abstract}
We present a unified general method for the asymptotic study of graphs from the so-called \lq\lq subcritical\rq\rq$ $ graph classes,
which include the classes of cacti graphs, outerplanar graphs, and series-parallel graphs.
This general method works both in the labelled and unlabelled framework.
The main results concern the asymptotic enumeration and the limit laws of properties of random graphs chosen from subcritical classes.
We show that the number $g_n/n!$ (resp. $g_n$) of labelled (resp. unlabelled) graphs on $n$ vertices
from a subcritical graph class ${\cG}=\cup_n {\cG_n}$ satisfies asymptotically the universal behaviour
$$
g_n\ \! =\ \! c\ \!n^{-5/2}\ \!\gamma^n\ \!  (1+o(1))
$$
for computable constants $c,\gamma$, e.g. $\gamma\approx 9.38527$ for unlabelled series-parallel graphs,
and that the number of vertices of degree $k$ ($k$ fixed) in a graph chosen uniformly at random from $\cG_n$, converges (after rescaling)
 to a normal law as $n\to\infty$.
\end{abstract}

\maketitle

\section{Introduction}
Several enumeration problems on classes of labelled planar structures, e.g. labelled planar graphs, were solved recently~\cite{BoGiKaNo07,BKLM07,GiNo09}. These results were successfully used for efficient random generators for labelled planar graphs~\cite{Fu07a}, based on Boltzmann samplers~\cite{DuFlLoSc04}. In contrast, less is known about enumerative results on classes of \emph{unlabelled} planar structures: the only classes treated so far are forests \cite{Otter48} and more recently outerplanar graphs~\cite{BoFuKaVi07b}.

In this paper we present a general framework to enumerate in a unified way a wide variety of labelled and unlabelled classes of graphs. Our main contribution is a universal method for rich classes of \emph{unlabelled} graphs, the so-called \lq\lq subcritical\rq\rq$\!$ classes of graphs, which is established through the novel singularity analysis of counting series and yields asymptotic estimates and limit laws for various graph parameters. In order to make this method more accessible and transparent, we include a brief analysis of the corresponding labelled classes, which was already carried out in~\cite{GiNoRu09}.

Another main contribution is asymptotic estimates and limit laws for various graph parameters of  unlabelled series-parallel graphs. We study the class of unlabelled series-parallel graphs, firstly as an important subclass of planar graphs whose  asymptotic study has not been carried out so far and therefore it is interesting in its own right, and secondly as a concrete prototype-example to illustrate how our general method is applied.
A graph is \emph{series-parallel} (SP-graphs for conciseness) if its $2$-connected components are obtained from a single edge by recursive subdivision of edges (series operation) and duplication of edges (parallel operation). Equivalently, SP-graphs can be defined in terms of minors as graphs
which exclude $K_4$ as a minor. Finally, a graph is series-parallel if and only if its tree-width is at most $2$. Applying our general method, we show that the number ${g}_n$ of unlabelled SP-graphs on $n$ vertices is asymptotically of the form
$$
{g}_n\ \! =\ \!  c\ \! n^{-5/2}\ \! \rho^{-n}\ \! (1+o(1)),
$$
where $\rho\approx 0.10655$. Let ${G}_n$ be a graph chosen uniformly at random among all unlabelled SP-graphs on $n$ vertices. The random variable $X_n$ counting the number of edges (blocks, or cut-vertices) in ${G}_n$ features a central limit law
\[
\frac{X_n - \mathbb{E}\, X_n}{\sqrt{\mathbb{V}{\rm{ar}}\, X_n}} \to N(0,1),
\]
where $\mathbb E\, X_n = \mu n+O(1)$ and $\mathbb V{\rm{ar}}\, X_n= \sigma^2 n+O(1)$ for computable constants $\mu$ and $\sigma^2$.
In addition, the random variable $X_n^k$ counting the number of vertices of degree $k$ (for $k$ fixed) in ${G}_n$ satisfies a central limit law with mean $\mathbb E\, X_n^k= \mu_k n+O(1)$ and variance $\mathbb V{\rm{ar}}\, X_n^k=\sigma_k^2 n+O(1)$ where $\mu_k$ and $\sigma_k^2$ are computable constants. We also present a simple and general condition assuring $\sigma^2>0$, which holds in a wide variety of graph families.

Furthermore we show that the same subexponential term $n^{-5/2}$ appears in other classes of
graphs (in both labelled and unlabelled cases), which is, roughly speaking, inherited from a tree-like nature. This behaviour appears as a consequence
of a \emph{subcritical composition} scheme which appears in the specification of the counting series associated to connected graphs of the class. Such classes of graphs arising from subcritical composition scheme are called \emph{subcritical} classes of graphs, whose formal definition is rather technical and therefore will be suspended to Sections 4 and 5.

We consider block-stable classes of graphs: we say a class $\mathcal{G}$ of graphs \emph{block-stable} if and only if for each graph $G\in \mathcal{G}$, each of its $2$-connected components (also called \emph{blocks}) belongs to $\mathcal{G}$. The class of cacti graphs, the class of outerplanar graphs, the class of SP-graphs, and other classes of graphs defined in terms of a class of $2$-connected components are block-stable. Additionally, we consider classes of graphs which are defined by a \emph{finite} set of $3$-connected graphs. Observe that SP-graphs can be seen as graphs without 3-connected components. We show that the classes of cacti graphs, outerplanar graphs, and SP-graphs, are subcritical and prove that the asymptotic estimates and limit laws for graph parameters of subcritical classes of graphs follow the same asymptotic pattern and limit laws as the class of SP-graphs.

The asymptotic study of subcritical classes of graphs consists of two steps: formal and analytic steps. The formal step consists in translating
Tutte's seminal ideas on decomposing graphs into components of higher connectivity~\cite{BT64,Tutte,Tu73} in terms of the \emph{decomposition grammar}, which is comparable to the ones introduced in~\cite{ChFuKaSh07} (see also~\cite{GLLW08}). This decomposition grammar translates combinatorial conditions into functional equations satisfied by the counting series of various classes of graphs. These counting series depend on the connectivity degree and the way how the graphs are rooted. In the analytic step, we extract \emph{singular expansions} of the counting series from the systems of functional equations. The main ingredient in this step is from~\cite{drmota97systems}, in which precise singular expansions are deduced for very general systems of functional equations. Finally, we derive asymptotic formulas from these singular expansions by extracting coefficients, based on the transfer theorems of singularity analysis~\cite{flajolet,FlaSe}.

We also study natural parameters on a uniform random graph on $n$ vertices chosen from a subcritical class, such as
the number of edges, the number of blocks and the number of cut-vertices. These parameters satisfy normal limit laws, which come from the \emph{additivity} behaviour of these parameters. We compute their expectation and variance, which characterise completely their limit
distribution. These parameters were studied in~\cite{GiNoRu09} for various labelled classes of graphs. In this paper we rediscover these results for labelled case and obtain new results for unlabelled case. Finally, applying the general techniques for systems of functional equations we deduce the limit law for the degree distribution of a graph chosen uniformly at random among all the graphs on $n$ vertices, by analysing in a unified way both labelled and unlabelled subcritical classes. The present work complement the work~\cite{BePaSt09,DrGiNo07,DrGiNo08} for related problems on labelled graphs. Other \emph{non-additive} parameters, such as the size of the \emph{largest} block are not treated here. This parameter was studied in the labelled framework in~\cite{GiNoRu09} and~\cite{PaSt09}. In the latter, the theory of Boltzmann samplers is applied.
\\\\
\textbf{Outline of the paper.} The paper is organised as follows. In Section~\ref{sec:pre} we introduce the notation and the terminology used in this paper. All the analytic machinery needed in order to deal with systems of functional equations, extraction of coefficients and limit laws is introduced in Section~\ref{sec:tools}. In this section, we adapt the results from~\cite{drmota97systems} to our context and recall the transfer theorems of singularity analysis from~\cite{FlaSe} and other key results necessary to derive limit laws for graph parameters. We also introduce a novel result which assures the positivity of the variance under certain easy conditions. In Section~\ref{sec:sub_labelled} we obtain general results for the enumeration of labelled subcritical graph classes. These results are generalised to unlabelled subcritical graph classes in Section~\ref{sec:sub_unlabelled}. These enumerative results we obtain are
consequences of the general framework presented in Section~\ref{sec:tools}. Concrete examples are studied in Section~\ref{sec:ex}. This section includes the analysis of unlabelled cacti graphs, unlabelled outerplanar graphs, and unlabelled SP-graphs. We also obtain general enumerative results for classes of graphs
defined by a {finite} set of $3$-connected graphs. Limit laws are studied in Section~\ref{sec:limitlaws}. All parameters studied in this section give rise to normal distributed random variables, independently of the class (in either labelled or unlabelled setting). The degree distribution, a technically more involved and interesting parameter, is studied in Section~\ref{sec:degdis}. Finally, the constant growth for unlabelled SP-graphs is computed in Section~\ref{sect: constant}, using a numerical method.

\section{Graph classes, block-decomposition, and counting}\label{sec:pre}
\subsection{Combinatorial classes and counting series}\label{sec:def_series}
As described in~\cite{BeLaLe,FlaSe}, a \emph{labelled combinatorial class} is a set $\cG=\cup_{n\geq 0}\,\cG_n$ of objects such that $\cG_n$ is finite for $n\geq 0$. Each object $g$ in $\cG_n$ has $n$ labelled \lq\lq atoms\rq\rq (e.g., vertices of a graph)  carrying distinct labels in $\{1,\dots,n\}$ and $n$ is called the \emph{size} of $g$.
Two objects of $\cG$ (of the same size $n$) are called \emph{isomorphic} if one is defined
from the other by relabelling. It is always assumed that a combinatorial class is stable
under relabelling, e.g., for a graph class $\cG$ we assume that a graph $g$ is in $\cG$
if and only if all graphs isomorphic to $g$ are also in $\cG$. This way the symmetric
group $\frak{S}_n$ acts on $\cG_n$: for  $\sigma\in\frak{S}_n$ and $g\in\cG_n$, $\sigma\cdot g$ has the same vertex and edge set as $g$, but each label $i$ in $g$ is replaced by $\sigma^{-1}(i)$ in $\sigma\cdot g$ and we write $\sigma\cdot g\equiv g$.
The set of objects of $\cG_n$ considered up to isomorphism
is denoted by $\tcG_n$ (in other words $\tcG_n=\cG_n/\frak{S}_n$), and the combinatorial class $\tcG:=\cup_n\, \tcG_n$
is called the \emph{unlabelled combinatorial class} associated with $\cG$.
For counting purpose one classically considers the exponential generating function (shortly the EGF)
in the labelled setting:
$$
\cG(z):=\sum\nolimits_{n\geq 0}\frac1{n!}|\cG_n|z^n,
$$
and the ordinary generating function (OGF) in the unlabelled setting:
$$
\tcG(z):=\sum\nolimits_{n\geq 0}|\tcG_n|z^n.
$$
For unlabelled enumeration, it proves convenient to consider a refinement of the OGF,
called the cycle-index sum, a series in infinitely many variables $s_1,s_2,\ldots$ defined
as
$$
Z_{\cG}(s_1,s_2,\ldots):=\sum\nolimits_{n\geq 0}\frac1{n!}\sum_{\substack{(\sigma,g)\in\frak{S}_n\times\cG_n\\ \sigma\cdot g=g}}w_{\sigma},
$$
where $w_{\sigma}:=s_1^{c_1}s_2^{c_2}\cdots s_n^{c_n}$ is the weight-monomial
of a permutation $\sigma$ of cycle type $1^{c_1}2^{c_2}\cdots n^{c_n}$ (that is, $c_i$ cycles of length $i$ for $i=1,\dots,n$).
The OGF is recovered from a specialisation of the cycle index sum~\cite{Ha} by replacing $s_i$ by $z^i$ for each $i\ge 1$:
$$
\tcG(z)=Z_{\cG}(z,z^2,z^3,\ldots).
$$
We will consider classes of graphs of various types depending on whether one marks vertices or not. All graphs are assumed to be simple (no loops nor multiple edges) and are labelled at vertices. A (vertex-)\emph{rooted graph} is a graph with a distinguished (labelled) vertex. A \emph{derived graph} or \emph{pointed graph} is a graph where one vertex is distinguished but not labelled (the other $n-1$ vertices have distinct labels in $\{1,\dots,n-1\}$).
Isomorphisms between two pointed graphs (or between two derived graphs) have to respect the distinguished vertex.

Given a graph class $\cG$, the \emph{rooted} class $\cGp$ is the class of rooted graphs from $\cG$, and the \emph{derived} class $\cG'$ is the class of derived graphs from $\cG$; since $|\cG'_{n-1}|=n|\cG_n|$ and $|\cGp_n|=n|\cG_n|$, we have respectively $\cG'(z)=\frac{\mathrm{d}}{\mathrm{d}z}\cG(z)$ and $\cGp(z)=z\cG'(z)$.

\subsection{Block-decomposition of a graph}\label{subsec:block_stable}
For $k\geq 0$, a graph is $k$-connected if one needs to delete at least $k$ vertices to disconnect it. Obviously, a graph $G$ is a set of its connected components. For the decomposition from connected graphs into 2-connected graphs we use the block structure of a connected graph. A block of a graph $G$ is a maximal 2-connected induced subgraph of $G$. We say a vertex of $G$ is incident to a block $B$ of $G$ if it belongs to $B$. The block structure of $G$ yields a bipartite tree with the vertex set consisting of two types of nodes, i.e. cut-vertices and blocks of $G$, and the edge set describing the incidences between the cut-vertices and blocks of $G$. This suggests a natural decomposition of connected graphs into 2-connected graphs and this holds also for rooted graphs. The root-vertex $v$ of a rooted graph $G$ is incident to a set of blocks and to each non-root vertex on these blocks is attached a rooted connected graph. In other words, a rooted connected graph rooted at $v$ is uniquely obtained as follows: take a set of derived 2-connected graphs and merge them at their pointed (distinguished but not labelled) vertices so that $v$ is incident to these derived 2-connected graphs, then replace each non-root vertex $w$ in these blocks by a rooted connected graph rooted at $w$ (which is allowed to consist of a single vertex and in this case it has no effect).

Through the entire paper, given a class of graphs $\cG$, we denote by $\cC$ (resp. $\cB$) the subfamily of connected (resp. $2$-connected) graphs in $\cG$. In the language of symbolic combinatorics from~\cite{BeLaLe,FlaSe}, the block-decomposition described above translates into the fundamental equations:
\beq\label{eq:GC}
\cG=\Set(\cC), \eeq
\beq\label{eq:CpBp}
\cCp=\cZ\cdot\Set(\cB'\circ\cCp),
\eeq
where the factor $\cZ$ in the last equation takes account of the root vertex (which is distinguished  and labelled),
the symbol $\cdot$ denotes the partitional product on combinatorial classes, and the symbol $\circ$ denotes substitution at an atom (see~\cite{BeLaLe} for definitions). As shown in Table~\ref{fig:dict}, there is a well-known dictionary~\cite{BeLaLe, FlaSe}, both in the labelled and in the unlabelled setting,
that translates equations relating combinatorial classes into equations relating the associated counting series.

\begin{table}[htb]
\begin{center}
\begin{tabular}{c|c|c|c}
Construction & Class & Labelled setting & Unlabelled setting \\
\hline
\hline
Sum & $\cC=\cA+\cB$ & $\cC(z)=\cA(z)+\cB(z)$& $\tcC(z)=\tcA(z)+\tcB(z)$\phantom{$\widetilde{\cC^{\cC}}$}\\[.1cm]
Product & $\cC=\cA\cdot\cB$ & $\cC(z)=\cA(z)\cdot\cB(z)$ & $\tcC(z)=\tcA(z)\cdot\tcB(z)$\\[.1cm]
Set & $\cC=\Set(\cB)$ & $\cC(z)=\exp(\cB(z))$ & $\tcC(z)=\exp\big(\sum_{i\geq 1}\tfrac1{i}\tcB(z^i)$\big)\\[.2cm]
Substitution & $\cC=\cA\circ\cB$ & $\cC(z)=\cA(\cB(z))$ & $\tcC(z)=Z_{\cA}(\tcB(z),\tcB(z^2),\ldots)$
\end{tabular}
\end{center}
\caption{The dictionary that translates combinatorial constructions into operations on counting series.}
\label{fig:dict}
\end{table}
A graph class $\cG$ is called \emph{block-stable} if it contains
the link-graph $\ell$, which is a graph with one edge together with its two (labelled) end vertices, and satisfies the property that a graph $G$ belongs to $\cG$ if and only if all the blocks of $G$ belong to $\cG$. Block-stable classes include classes of graph specified by a finite list of forbidden minors that are all 2-connected, for instance, planar graphs ($\Forbid(K_5,K_{3,3})$), series-parallel graphs ($\Forbid(K_4)$), and outerplanar graphs ($\Forbid(K_4,K_{3,2})$).
For a block-stable graph class, \eqref{eq:GC} and~\eqref{eq:CpBp} translates into equations of EGFs in the labelled setting:
$$
\cG(z)=\exp(\cC(z)),
$$
$$
\cCp(z)=z\exp(\cB'(\cCp(z))),
$$
and into equations of OGFs in the unlabelled setting:
$$
\tcG(z)=\exp\left(\sum\nolimits_{i\geq 1}\tfrac1{i}\tcC(z^i)\right),
$$
$$
\tcCp(z)=z\exp\left(\sum\nolimits_{i\geq 1}\tfrac1{i}Z_{\cB'}(\tcCp(z^i),\tcCp(z^{2i}),\tcCp(z^{3i}),\ldots)\right).
$$
A refined version of the last equation will turn out to be useful later, which expresses \eqref{eq:CpBp} in terms of the cycle index sum:
\beq\label{eq:cyc_C_B}
Z_{\cCp}(s_1,s_2,\ldots)=s_1\exp\left(\sum\nolimits_{i\geq 1}\tfrac1{i}Z_{\cB'}(Z_{\cCp}(s_i,s_{2i},\ldots),Z_{\cCp}(s_{2i},s_{4i},\ldots),\ldots)\right).
\eeq

With these systems of equations of EGFs, OGFs and the cycle index sum, we will perform  singularity analysis under certain general conditions -- so-called \lq\lq subcriticallity\rq\rq$\!$ conditions -- in order to get asymptotic results (asymptotic enumeration and limit laws for graph parameters). The tools for singularity analysis are described in the next section.

\section{Tools for the asymptotic analysis}\label{sec:tools}

The purpose of this section is to collect several facts on solutions of functional equations.
Most of the following properties (and proofs) can be found in \cite{Drmota,FlaSe},
which serve as general references for this subject.
For our purposes we, however, need to adjust several points  with additional properties.

\subsection{Singular expansions of multivariate series}\label{sec:singularexpansion}
We consider a power series of the form
$y=y(z,v_1,\ldots,v_k)$ in general with nonnegative coefficients,
where $z$ is singled out as the \emph{primary variable} and $v_1,\ldots,v_k$
are secondary variables (possibly there is none). From now on
we use the abbreviation $\vv$ for $(v_1,\ldots,v_k)$.
A valuation $\vvo$ of $\vv$ is called \emph{admissible} if all components of $\vvo$
are positive and if $y(z,\vvo)$ is a valid power series in $z$, i.e., $[z^n]y(z,\vvo)<\infty$
for each $n\geq 0$.

Consider a fixed positive valuation $(z_0,\vvo)$ of $(z,\vv)$. Then $y$ is said to have a \emph{square-root expansion} around $(z_0,\vvo)$
if $\vvo$ is an admissible valuation, $z_0$ is the radius of convergence of $y(z,\vvo)$,  and the representation
$$
y(z,\vv)=a(z,\vv)-b(z,\vv)\sqrt{1-z/\rho(\vv)}
$$
holds in a neighbourhood of $(z_0,\vvo)$ (except in the part where $1-z/\rho(\vv)\in\mathbb{R}^-$),
where the functions $a(z,\vv)$ and $b(z,\vv)$ are analytic at $(z_0,\vvo)$, $b(z_0,\vvo)>0$, $\rho(\vv)$ is analytic at $\vvo$, and $\rho(\vvo)=z_0$.
The function $\rho(\vv)$ in the expansion
is called the \emph{singularity function} of $y$
relative to $z$.

Moreover, we will be particularly interested in functions $y(z,\vv)$ with such a singular behaviour,
where $\rho(\vv)$ is the only singularity on the circle $|z| = |\rho(\vv)|$
(if $\vv$ varies in a suitable neighbourhood of $\vv_0$) and
$y(z,\vv)$ can be analytically continued to the region $\{ z\in \mathbb{C} : |z| < |\rho(\vv)|+\epsilon,\
1 - z/\rho(\vv) \not\in \mathbb{R}^-\}$ (for some $\epsilon>0$ that
is uniform in this neighbourhood of $\vv_0$).
In this case one can use an asymptotic transfer principle
by Flajolet and Odlyzko \cite{flajolet}  (see also Section~\ref{sec:transfer})
to obtain asymptotics for the coefficient $[z^n]\, y(z,\vv)$ of
the form
$$
[z^n]\, y(z,\vv) = \frac{b(\rho(\vv),\vv)}{2\sqrt \pi}\,  n^{-3/2}\, \rho(\vv)^{-n}\, \left( 1 + O\left( n^{-1} \right) \right).
$$
Similarly, $y$ is said to have a singular expansion of order $3/2$ at $(z_0,\vvo)$ if the expansion is of the form
$$
y(z,\vv)=a(z,\vv)+b(z,\vv)\cdot(1-z/\rho(\vv))^{3/2},
$$
with otherwise the same conditions as for square-root expansions.

It is not difficult to show that if a function $y(z,\vv)$ admits a square-root expansion around $(z_0,\vvo)$, then the function $\int_0^z\, y(x,\vv)\, \mathrm{d}x$ admits a singular expansion of order $3/2$ at $(z_0,\vvo)$ (see \cite{DrGiNo08}).

\subsection{Singularity analysis of systems of functional equations}
\label{sec:systemequations}
Consider a vector $\vy=(y_1,\ldots,y_r)$ of formal power series in the formal variables $z,\vv$, with $\vv=(v_1,\ldots,v_k)$ which is a solution of an equation-system of the form
$$
\rm{(E)}:\qquad \vy=\vF(\vy;z,\vv),
$$
where $\vF(\vy;z,\vv)=(F_1(\vy;z,\vv),\ldots,F_r(\vy;z,\vv))$ are power series with nonnegative coefficients. In order to avoid trivial situations as in equations like $y = zy^2$, where $y(z) = 0$ is the only solution, we will assume that $\vy = {\bf 0}$ is {\bf not} a solution. We call such a system \emph{positive system}.

The \emph{Jacobian matrix} $\JacF$ of the system is the $r\times r$ matrix whose $(i,j)$-coefficient is $\partial F_j/\partial y_i$. The \emph{singularity equation} associated with $\rm{(E)}$ is
$$
\rm{(S)}:\qquad 0=\mathrm{Det}\left(\mathbf{I_r}-\JacF\right),
$$
where $\mathbf{I_r}$ is the $r\times r$ identity matrix. The \emph{singularity system} of ($\rm{E}$) is the system $\{\rm{(E)},\ \rm{(S)}\}$, which has $r+1$ equations (the first $k$ ones for ($\rm{E}$), the last one for ($\rm{S}$)).

The \emph{dependency graph} of $\vF$ is the directed graph on $V=\{1,\ldots,r\}$ such that there is an edge from $i$ to $j$ if and only if the $j$th component $F_j$ of $\vF$ really involves $y_i$, that is, the power series $\partial F_j/\partial y_i$ is not $0$. The system $\rm{(E)}$ is called \emph{strongly recursive} if the dependency graph of $\vF$ is strongly connected (which means that every pair of vertices is linked by a directed path).

Informally this condition says that  no subsystem of $\rm{(E)}$ can be solved prior to the whole system of equations. An equivalent condition is that the corresponding adjacency matrix and, thus, the Jacobian matrix  $\JacF$ is irreducible. The most important property of irreducible matrices ${\bf A}$
with non-negative entries is the Perron-Frobenius Theorem (see \cite{Minc}) saying that there is a unique positive and simple eigenvalue $\lambda_{max}=\lambda_{max}({\bf A})$ with the property that all other eigenvalues $\lambda$ satisfy $|\lambda|\le \lambda_{max}$. This unique positive
eigenvalue is a strictly increasing function of the entries of the non-negative matrix. More precisely, if ${\bf A} = (a_{ij})$ and ${\bf A}' = (a_{ij}')$ are different irreducible non-negative matrices with $a_{ij}\le a_{ij}'$ (for all $i,j$) then $\lambda_{max}({\bf A})\le \lambda_{max}({\bf A}')$.
Moreover, every principal submatrix has a smaller dominant eigenvalue.

Consider an admissible valuation $\vvo$ of $\vv$ and assume that the radius of convergence $z_0$ of $z\to y_1(z,\vvo)$ is finite and strictly positive.
If the system $\rm{(E)}$ is strongly recursive then it is easily shown that, for each $i\in \{1,\ldots,r\}$, $z_0$ is the radius of convergence of $y_i$
(due to the dependency graph being strongly connected) and $y_i(z,\vvo)$ converges to a finite constant $\tau_i>0$ as $z\to z_0$ (due to the fact that $\vF$ is nonlinear according to $\vy$).

Note also that $\rm{(S)}$ says that $\lambda = 1$ is an eigenvalue of $\JacF$. We observe that if $\vy=\vy(z,\vvo)$ is an analytic solution of
a strongly recursive system $\rm{(E)}$ that is singular at $z=z_0$ and $\vy_0 = \vy(z_0,\vvo)$ is a finite vector,
then $\rm{(S)}$ is satisfied for $(y;z,\vv)=(y_0;z_0,\vvo)$ provided that $(y_0;z_0,\vvo)$ is an inner point of the region of convergence
of $\vF$. (This explains the term singularity system.) However, in order to obtain the
radius of convergence (in positive and analytically well-founded systems) we will need the condition that
$\lambda_{max}(\JacF) = 1$ (see \cite{BBY09}).

\begin{definition}\label{Def3}
A system $\rm{(E)}$ is called \emph{analytically well-founded}
at a fixed positive valuation $\vvo$ if the following conditions
are satisfied:
\be
\item
The valuation $\vvo$ is admissible for the system of power series
$\vF(\vy;z,\vvo)$, we have
$[\vy^{\bf m} z^n]\, F_i(\vy;z,\vvo)\ge 0$  for all $i$,
and $[z^n]\vF({\bf 0};z,\vvo) \ne {\bf 0}$.
\item The function $\vF$ is not affine in $\vy$ and depends on $z$,
that is, there are $i$ and $j$ in $\{1,\ldots,r\}$
such that $\partial^2\vF/\partial y_i\partial y_j\neq 0$
and  $\partial\vF/\partial z\neq 0$,

\item
There exist $z_0> 0$ and a positive vector $\vy_0$ for which
$(\vy_0;z_0,\vvo)$ is an inner point of the
region of convergence of $F_i(\vy;z,\vv)$ (for $i\in \{1,\ldots,r\}$) and
$\rm{(E)}$ and $\rm{(S)}$ are satisfied for $(\vy;z,\vv)=(\vy_0;z_0,\vvo)$
and that all eigenvalues $\lambda$ of the Jacobian matrix
$\JacF(\vy_0;z_0,\vvo)$ satisfy $|\lambda|\le 1$.
\ee
\end{definition}

It is not clear that the singularity system
$\rm{(E)}$, $\rm{(S)}$ has a proper solution. However, if
$\rm{(E)}$ is a positive and strongly recursive system and
if $\vF(\vy;z,\vv)$ is a vector of entire functions in
$\vy$ and $z$ (for $\vv$ in a neighbourhood of $\vv_0$)
then it is always solvable. In particular it is enough to
consider the singularity of the power series solution of
$\rm{(E)}$. We also get the property
$\lambda_{max}(\JacF(\vy_0;z_0,\vvo)) = 1$.

Note that the radius of convergence $z_0$ of the solution $\vy$
of a positive and strongly recursive system $\rm{(E)}$
is always finite. Furthermore, $\vy(z_0)$ is finite, too, if
$\vF$ is not affine in $\vy$.

The following theorem contains the main properties of
systems of equations $\rm{(E)}$ that we will use in the sequel
(for a proof see \cite{Drmota}).

\btheo\label{Thsystem1}
Suppose that an equation $\rm{(E)}$ is \emph{strongly recursive} and
\emph{analytically well-founded}
at a valuation $\vvo$. Then there is a unique vector of
power series $\vy = \vy(z,\vv)$
in the variables $z,\vv$ that satisfies $\rm{(E)}$.
Furthermore, the components of $\vy$ have non-negative coefficients
$[z^n]\, y_i(z,\vvo)$ (for $i\in \{1,\ldots,r\}$)
and a square-root expansions around $(z_0,\vvo)$.

Moreover, if $[z^n]\,y_1(z,\vvo) > 0$ for all $n\ge n_0$
then $z_0$ is the only singularity on the radius of convergence
$|z| = z_0$ and all components $\vy$ can be analytically continued to
the region $D = \{ z\in \mathbb{C} : |z| < |\rho(\vv)|+\epsilon,\
1 - z/\rho(\vv) \not\in \mathbb{R}^-\}$, where $\epsilon>0$ is uniform
for $\vv$ in some neighbourhood of $\vvo$.
\etheo

The condition $[z^n]\,y_1(z,\vvo) > 0$ (for $n\ge n_0$)
is usually verified by
using a combinatorial interpretation of the coefficients.
In the case of a single equation $y = F(y;z,v) = \sum_{n,m,k}
a_{n,m,k} z^ny^mv^k$ it is also possible to check this
with the help of conditions on the coefficients $a_{n,m,k}$.
For example, if there exist $m>1$, $n_1,n_2,n_3$ and
$k_1,k_2,k_3$ with $a_{n_1,m,k_1} \ne 0$,  $a_{n_2,m,k_2} \ne 0$,
$a_{n_3,m,k_3} \ne 0$ such that $n_2-n_1$ and $n_3-n_1$ are
coprime or if there are $n_1,n_2$, $m_1>1$, $m_2>1$, $k_1,k_2$
with  $a_{n_1,m_1-1,k_1} \ne 0$,  $a_{n_2,m_2-1,k_2} \ne 0$
such that $n_1(m_2-1) - n_2(m_1-1) = 1$ then it also
follows that $[z^n]\,y(z,v_0) > 0$ (for $n\ge n_0$)
-- compare it with \cite{MeirMoon} and the methods used in the proof
of Lemma~\ref{Lepositive}.

\subsection{Transfer theorems of singularity analysis and central limit theorems}
\label{sec:transfer}
As detailed in the book by Flajolet and Sedgewick~\cite{FlaSe},
the singular  behaviour of a counting series can be translated to an asymptotic estimate
of the counting sequence, by coefficient extraction (which is done via contour
integrals in the complex plane).

Let $z_0$ be a non-zero complex number, and $\epsilon$ and $\delta$ positive (real) numbers. Then the region
\[
\Delta = \Delta(z_0,\epsilon,\delta) = \{z \in \mathbb{C}:
 |z|<z_0+\epsilon,\, |\arg(z/z_0-1)|>\delta \}
\]
is called a $\Delta$-region. The basic observation (see \cite{flajolet})
is that a singular expansion around the singularity $z_0$ that is
uniform in a $\Delta$-region transfers directly to an asymptotic
expansion for the coefficients.
Suppose that a function $y(z)$ is analytic in a region $\Delta$-region
$\Delta(z_0,\epsilon,\delta)$ and satisfies
\[
y(z)  =  C \left( 1 -  z/{z_0}\right)^{\alpha} +
O\left( \left( 1- z/{z_0}\right)^{\beta}\right), \qquad
z\in\Delta(z_0,\epsilon,\delta),
\]
where $\beta > \alpha$ and $\alpha$ is a non-negative integer. Then we have
\begin{equation}\label{eqanexp}
[z^n]\, y(z) = C \frac{n^{-\alpha-1}}{\Gamma(-\alpha)} z_0^{-n}
 + O\left(z_0^{-n} n^{\max\{-\alpha-2,-\beta-1\}}\right).
\end{equation}
It is an important additional observation that the implicit constants are also effective which means that the $O$-constant in the expansion of $y(x)$
provides explicitly an $O$-constant for the expansion for $[z^n]\, y(z)$, and that the same statement is true if we change the $O$-constants by $o$-constants.
See \cite{flajolet} for details.
In particular it follows that singular expansions that are uniform in some
parameter also translate into asymptotic expansions of the form
(\ref{eqanexp}) with a uniform error term.
In particular it applies for functions $y(z,\vv)$
with square-root expansion around $(z_0,\vvo)$ or with
singular expansion of order $3/2$, provided that they can be
analytically continued to a $\Delta$-region.

Next we restrict ourselves to univariate $v\in \mathbb{C}$ and
are interested in bivariate asymptotic expansions of the
coefficients $[z^nv^m]\, y(z,v)$ when $y$ has a square-root expansion.
We introduce the function $\mu$ defined as
\[
\mu(v) = -\frac{v\rho'(v)}{\rho(v)}.
\]
We call it \emph{regular} in a closed interval $\overline V$ of the
positive real line if $\mu(v)$ is strictly increasing on $\overline V$.
In this case $\nu = \mu^{-1}$ denotes the inverse function of $\mu$.
Furthermore, we set $\sigma(v) = \sqrt{v \mu'(v)}$ which is positive
in the regular case. If we additionally assume that $z_0 = \rho(v_0)$ is
the only singularity of $y(z,v)$ for $|z| = z_0$ and $|v|= v_0$ then
we have uniformly for
$m/n \in \nu(\overline V)$
\begin{align}
&[z^nv^m]\, y(z,v)  \label{eqLetransfer3}\\
&= \frac{b(\rho(\nu(m/n)),\nu(m/n))}{2\sqrt{2}\,\pi\,
(n/m)\sigma(\nu(m/n))\, mn} \rho(\nu(m/n))^{-n} \nu(m/n)^{-m}
\left( 1 + O\left( n^{-1}\right)\right). \nonumber
\end{align}
(For a proof we refer to \cite{DrmEJC}).
Again it is clear that this has a direct analogue
for functions with a singular expansion of order $3/2$.

It is relatively easy to check the condition that
$z_0 = \rho(v_0)$ is
the only singularity of $y(z,v)$ for $|z| = z_0$ and $|v|= v_0$
and also that $\mu$ is regular
-- compare it with the remark following the proof of Lemma~\ref{Lepositive}.

Suppose that $m = \mu n + O(\sqrt n)$.  The asymptotic expansion (\ref{eqLetransfer3}) for
the coefficient $[z^nv^m]\, y(z,v)$ behaves locally like
\begin{equation}\label{eqlocalCLT}
[z^nv^m]\, y(z,v) = \frac{b(\rho(1),1)}{2\sqrt{\pi}\, \sigma(1) n^2}
 e^{-(m-\mu(1)n)^2/(2 \sigma(1)^2 n)}\rho(1)^{-n}
\left( 1 + O(n^{-1/2})\right),
\end{equation}
which suggests that there is a central limit theorem behind.
Actually this is true.

Let $X_n$ be a random variable with probability distribution
$\mathbb{P}(X_n = m) = [z^nv^m]\, y(z,v)/ [z^n]\, y(z,1)$, then
the asymptotic expansion (\ref{eqlocalCLT}) is indeed a local limit theorem for $X_n$.
In general, there is a \emph{combinatorial central limit theorem}. For the sake of brevity we do not list all possible versions but
only for a single equation and we comment on systems of equations.
\btheo\label{ThcombCLT1}
Suppose that $X_n$ is a sequence of random variables whose probability generating
function is given by
\[
\mathbb{E}\, v^{X_n} = \frac{[z^n]\,y(z,v)}{[x^n]\, y(z,1)},
\]
where $y(z,v)$ is a power series that is the (analytic) solution of
the functional equation $y = F(y;z,v)$, where $F(y;z,v)$ satisfies the
assumptions of Theorem~\ref{Thsystem1}. In particular, let
$z=z_0>0$ and $y=y_0>0$ be the proper solution of the
system of equations $y = F(y;z,1)$, $1 = F_y(y;z,1)$
\footnote{For convenience we use the notation $F_y$ to denote the
partial derivative $\partial F/\partial y$.}.
Set
\begin{align*}
\mu &= \frac {F_v}{z_0F_z},\\
\sigma^2 &= \mu + \mu^2 +\frac {1}{z_0F_z^3F_{yy}}\Bigl( F_z^2(F_{yy}F_{vv} - F_{yv}^2)
-2F_zF_v(F_{yy}F_{zv}-F_{yx}F_{yv}) \\
&\quad \quad \ \ \ + F_v^2(F_{yy}F_{zz} - F_{yz}^2) \Bigr),
\end{align*}
where all partial derivatives are evaluated at the point $(y_0;z_0,1)$.
Then the asymptotic mean and variance of $X_n$ satisfy
\[
\mathbb{E}\, X_n = \mu n + O(1) \quad\mbox{and}\quad
\mathbb{V}{\rm{ar}}\, X_n = \sigma^2 n + O(1)
\]
and if $\sigma^2> 0$
\[
\frac{X_n - \mathbb{E}\, X_n}{\sqrt{\mathbb{V}{\rm{ar}}\, X_n}} \to N(0,1).
\]
\etheo
Note that
$\mu= \frac {F_v}{z_0F_z}$ is the same as the other $\mu$ defined above as $\mu(v) = -\frac{v\rho'(v)}{\rho(v)}$ when $v=1$. Similarly we have $\sigma^2 = \sigma^2(1)$.

The situation for a system of equations is quite similar
(even if we additionally consider a random vector ${\bf X}_n$ instead
of a random variable $X_n$).
Suppose that $\vy=(y_1,\ldots,y_r)$ is the solution of a system
of equations $({\rm E})$ and that the assumptions of Theorem~\ref{Thsystem1}
are satisfied. Furthermore set $y(z,\vv) = H(\vy(z,\vv);z,\vv)$ for a
power series $F$ with non-negative coefficients, for which
$(\vy_0;z_0,{\bf 1})$ is inner point of the region of convergence and
we have $H_{\vy}(\vy_0;z_0,{\bf 1}) \ne 0$. Then the random vector
${\bf X}_n = (X_{1;n},\ldots,X_{k;n})$ with probability generating
function
\[
\mathbb{E}\, \vv^{{\bf X}_n} =
\mathbb{E}\, v_1^{X_{1;n}} \cdots v_k^{X_{k;n}} =
\frac{[z^n]\, y(z,\vv)}{[z^n]\, y(z,1)}
\]
is asymptotically normal with asymptotic mean
$\mathbb{E}\,{\bf X}_n = n \mbox{\boldmath{$ \mu $}} + O(1)$ and covariance matrix
$\mathbb{C}{\rm {ov}}\,{\bf X}_n = n{\bf \Sigma} + O(1)$, where
\begin{equation} \label{meanformula}
\mbox{\boldmath{$ \mu $}} = \frac 1{z_0}\frac {{\bf b^T}
{\bf F}_{\bf v}(\vy_0;z_0,{\bf 1})}
{{\bf b^T} {\bf F}_x(\vy_0;z_0,{\bf 1})},
\end{equation}
in which ${\bf b}$ is (up to scaling) the unique positive left eigenvector
of $\JacF$, and ${\bf \Sigma}$ is a positive semi-definite matrix
which can be computed with the help of second derivatives
(for details see \cite{Drmota}). In many applications ${\bf b}$ appears to be $(1,\ldots,1)^{\bf T}$, which is due to the special structure of the systems of equations. The source of the central limit theorem is actually a singular expansion with singular term $\left(1-\frac{z}{\rho(v)}\right)^\alpha$ with $\alpha \notin \mathbb{N}$, and thus a central limit theorem with the same mean and variance follows also for generating functions given by $\int y(z,\vv) = \int F(\vy(z,\vv);z,\vv)$, which have the same singularity, but of order $\frac{3}{2}$.

Finally we comment on the positivity of $\sigma^2=\sigma^2(1)$ in the case of a single functional equation $y = F(y;z,v)$.
(Equivalently this concerns the question whether
$\mu(v) = vF_v(y(\rho(v),v);\rho(v),v)/(\rho(v)F_z(y(\rho(v),v);\rho(v),v))$
is regular in a neighbourhood of $v=1$.)

\begin{lem}\label{Lepositive}
Let $y = F(y;z,v)= \sum_{n,m,k} a_{n,m,k} z^ny^mv^k$ be an
analytically well founded equation for the valuation $v_0 = 1$
as given in Theorem~\ref{ThcombCLT1}.
Suppose that there are three triples $(n_j,m_j,k_j)$, $j = 1,2,3$, of
integers with $m_j> 0$, $j=1,2,3$,  and
\[
\left| \begin{array}{rrr}
n_1 & m_1-1 & k_1 \\ n_2 & m_2-1 & k_2 \\n_3 & m_3-1 & k_3
\end{array} \right| \ne 0
\]
such that $a_{n_j,m_j,k_j} \ne 0$, $j=1,2,3$.
Then $\sigma^2> 0$.
\end{lem}

\begin{proof}
Let $x = \rho(v)$ be the solution of the singular system
$y = F(y;z,v)$, $1 = F_y(y;z,v)$. We will first show that
$|\rho(e^{it})| > \rho(1)$ for real $t\ne 0$ that are sufficiently
small. This property will be then used to prove that $\sigma^2 > 0$.

First it is clear that $|\rho(e^{it})| \ge \rho(1)$ for all
real $t$ for which $\rho(e^{it})$ exists. For, if
$|\rho(e^{it})| < \rho(1)$ then we would have
$|y(\rho(e^{it}, e^{it})| < y(\rho(1),1)$ and also
\begin{equation}\label{eqFyest}
|F_y(y(\rho(e^{it}, e^{it}); \rho(e^{it}),e^{it})| <
F_y(y(\rho(1),1);\rho(1),1) = 1
\end{equation}
which is a contradiction. (Note that we have used here the
assumptions $F_{yy}\ne 0$ and $F_v \ne 0$.)

Now assume that $|\rho(e^{it})| \le \rho(1)$ for some real number $t$.
 Then an inequality similar to (\ref{eqFyest}) implies
that for all $n,m,k$
\[
m a_{n,m,k} z^{n}y^{m-1}v^k = m a_{n,m,k} z_0^{n}y_0^{m-1},
\]
where we used the abbreviations $z_0 = \rho(1)$,
$y_0 = y(\rho(1),1)$, $z = \rho(e^{it})$, $y = y(\rho(e^{it}, e^{it})$, and
$v = e^{it}$. In particular, it follows that
$z^{n_j}y^{m_j-1}v^{k_j} = z_0^{n_j}y_0^{m_j-1}$ for $j=1,2,3$.
Hence, if we set $z = z_0e^{ir}$, $y = y_0e^{is}$ (and $v = e^{it}$)
it follows that
\[
n_j r + (m_j-1) s + k_j t = 2\pi l_j \qquad (j = 1,2,3)
\]
for some integers $l_j$, $j =1,2,3$. This is a regular system and
implies that there is a (unique) solution of the form
$r = 2\pi L_1/M$, $s = 2\pi L_2/M$, $t = 2\pi L_3/M$ (for integers
$L_1,L_2,L_3,M$). Hence, if $t\ne 0$ is sufficiently close to $0$
then $|\rho(e^{it})| > \rho(1)$.

Next consider the Taylor series of the function
\[
g(t) = \log \rho (e^t) = \sum_{j=0}^\infty \frac{\kappa_j}{j!} t^j.
\]
By definition we have $\kappa_0 = \log \rho(1)$, $\kappa_1 = \mu$, and
$\kappa_2 = \sigma^2$. Note that this representation and
the general property $|\rho(e^{it})| \ge \rho(1)$ implies that
$\kappa_2 \ge 0$. Suppose that $\kappa_2 = 0$ and let $\ell_0\ge 3$ be
the smallest integer with $\kappa_{\ell_0} \ne 0$.

We use now the fact that the assumptions of Theorem~\ref{ThcombCLT1}
imply that
\[
\mathbb{E}\, e^{tX_n} = \left( \frac{\rho(e^t)}{\rho(1)} \right)^{-n}
\left(1 + O\left( n^{-1} \right) \right).
\]
(This follows from the singular expansion of the solution $y(z,v)$ and
the asymptotic transfer results -- compare it with (\ref{eqanexp}) from above).
Hence, by using the Taylor expansion of $g(t)$ it follows that
\[
\mathbb{E}\, e^{t(X_n-\mu n)n^{-1/\ell_0}} =
e^{-\kappa_{\ell_0} t^{\ell_0}/\ell_0!} + O\left( n^{-1/\ell_0} \right).
\]
This means that the sequence of random variables
$Y_n = (X_n-\mu n)n^{-1/\ell_0}$ converges weakly (and we have convergence
of all moments) to a random variable $Y$ with Laplace transform
$\mathbb{E}\ e^{tY} = e^{-\kappa_{\ell_0} t^{\ell_0}/\ell_0!}$.
However, such a random variable that has a non-zero $\ell_0$-th moment
but zero variance does not exist. Hence, we finally have proved
$\sigma^2 > 0$.
\end{proof}

A slight variation of the above proof shows that if there
are three triples $(n_j,m_j-1,k_j)$, $j = 1,2,3$, with
determinant $\pm 1$ then $|\rho(e^{it})| > \rho(1)$ for
all $t \not \in 2\pi \mathbb{Z}$. This shows that $z_0 = \rho(1)$
is the only singularity of $y(z,v)$ for $|z| = z_0$ and $|v| = 1$.
This assumption can be used to obtain bivariate asymptotic of
the form (\ref{eqLetransfer3}).

\section{Subcritical graph classes: the labelled case}\label{sec:sub_labelled}
In this section  $\cG$ denotes always a block-stable class  of labelled graphs and $\cC$ (resp. $\cB$) its subclass consisting of connected (resp. $2$-connected) graphs.

\subsection{Definition of subcriticality}
Recall from Section~\ref{subsec:block_stable} that the EGFs of the block-stable class satisfy
$$\cG(z)=\exp(\cC(z)),\, \cCp(z)=z\exp(\cB'(\cCp(z))).$$
Given a series $g(y)$  it is easy to show that there is a unique series $f(z)$ that is a
solution of the equation
\beq\label{eq:fg}
f(z)=z\exp(g(f(z))).
\eeq
In addition $f(z)$ has nonnegative coefficients if $g(y)$ has nonnegative coefficients. Note that for a block-stable graph class, the solution of~\eqref{eq:fg}
is $f(z)=\cCp(z)$ when $g(y)$ is taken as $\cB'(y)$.

\begin{definition}
Let $g(y)$ be a series with non-negative coefficients such that $g(0)=0$. Let $f(z)$ be the unique
solution of~\eqref{eq:fg}. Let $\rho$ and $\eta$ be the radii of convergence of $z\mapsto f(z)$
and $y\mapsto g(y)$. Then the pair $(f(z), g(y))$ is called subcritical if \mbox{$f(\rho)<\eta$}.

A block-stable graph class $\cG$ with $\cC$ the connected subclass and $\cB$ the 2-connected subclass is called \emph{subcritical}
 if the pair $(\cCp(z),\cB'(y))$ is subcritical.
\end{definition}

Note that the singularity system for~\eqref{eq:fg} is
$$
y=z\exp(g(y)),\,  y\ \!g'(y)=1.
$$
In particular, for the pair $(\cCp(z),\cB'(y))$ the latter equation
rewrites to $y \cB''(y) = 1$. Hence, we have subcritiallity if
and only if $\eta \cB''(\eta) > 1$, compare with \cite{BePaSt09}.

\medskip

In what follows we will also consider functions  $f(z,v)$, $g(y,v)$ with
an additional {\it parameter $v$}. For example, suppose that we are
dealing with a bivariate generating function where the exponent of $v$ counts
the number of edges and the exponent of $z$ (or $y$) counts the
number of vertices. Suppose further that we already know that
the pair $(f(z,1), g(y,1))$ is subcritical. What can we say then
for the pair  $(f(z,v), g(y,v))$ if $v$ is sufficiently close to $1$?
Is there some {\it stability} of the subcriticallity?
Actually there is if the parameter that is counted by the
exponent of $v$ has a linear worst case behaviour in the
exponent $n$ of $z$ (or $y$). The essential consequence of the following
lemma is that radius of convergence
$z\mapsto f(z,v)$ and $y\mapsto g(y,v)$, respectively,
is continuous at $v= 1$. Hence we have
$f(\rho(v),v)<\eta(v)$ if $v$ is real and sufficiently close to $1$\footnote
{Note that it is sufficient to consider
real $v$ if we are just interested into (global) central
limit theorem and asymptotic results for moments.
Namely, in order to prove a theorem of the type of Theorem~\ref{ThcombCLT1}
one can work with the help of the Laplace transform $\mathbb{E}\, e^{tX_n}$
that is encoded by $\mathbb{E}\, e^{tX_n} = [z^n] A(z,e^t)/[z^n] A(z,1)$
when $\mathbb{P}(X_n = k) = a_{n,k}/a_n$.}.

\begin{lem}
Let $A(z,v) = \sum_{n,k\ge 0} a_{n,k} z^n v^k$ be a
power series with non-negative coefficients $a_{n,k}$ with the
property that $a_{n,k} = 0$ for $k > C n$ for some constant $C> 0$.
Let $R(v)$ denote the radius of convergence of the mapping
$z\mapsto A(z,v)$. Then we have for real $v>0$
\[
R(1) \min\{ 1, v^{-C} \} \le R(v) \le R(1) \max\{ 1, v^C\}.
\]
\end{lem}

\begin{proof}
If $v\ge 1$ then
\[
\sum_k a_{n,k} \le  \sum_{k} a_{n,k} v^k \le
\left( \sum_k a_{n,k} \right) v^{Cn}
\]
and consequently $R(v) \ge R(1) v^{-C}$.
Similarly we argue for $0< v \le 1$.
\end{proof}

\subsection{Asymptotic estimate for a subcritical graph classes}

We start with a quick analysis of subcritical graph classes and derive
their asymptotic number, compare it with~\cite{BePaSt09}.
\begin{lem}
Let $\cG$ be a labelled subcritical block-stable graph class with $\cC$ the connected subclass
and $\cB$ the 2-connected subclass. Then $\cCp(z)$ has a square-root singular expansion around its radius of convergence $\rho$. Furthermore, if
$[z^n]\, \cCp(z)>0$ for $n\ge n_0$ then $\rho$ is the only singularity
on the circle $|z| = r$ and $\cCp(z)$ can be continued analytically
to the region $D = \{ z\in \mathbb{C} : |z| < \rho+\epsilon,\
1 - z/\rho \not\in \mathbb{R}^-\}$ for some $\epsilon> 0$.
\end{lem}
\begin{proof}
The function $y=\cCp(z)$ is a solution of
$$
y=F(y;z),\, \mathrm{with}\,\, F(y;z)=z\exp(\cB'(y)).
$$
Let $\rho$ and $\eta$ be respectively the radii of convergence of $\cCp(z)$
and of $\cB'(y)$, and let $\tau:=\cCp(\rho)$.
Since $\cG$ is subcritical, we have $\tau<\eta$, hence $F(y;z)$ is analytic
at $(\tau,\rho)$. We conclude from Theorem~\ref{Thsystem1} that $\cCp(z)$ has a
square-root expansion at $\rho$ and that $\cCp(z)$ can be continued analytically
to $D$.
\end{proof}
\begin{thm}
Let $\cG$ be a subcritical block-stable graph class with
the property that $[z^n]\, \cCp(z)>0$ for $n\ge n_0$. Then
there exist constants $\gamma\geq e \approx 2.71828$
and $c>0$ such that
\beq\label{eq:estim_Gn}
[z^n]\cG(z)=c\ \!n^{-5/2}\ \!\gamma^n(1+o(1)) \ \ \mathrm{as}\ n\to\infty.
\eeq
\end{thm}
\begin{proof}
The function $\cC(z)$ satisfies
$$\cC(z)=\int_0^z\cCp(t)\frac{\mathrm{d}t}{t},$$
hence $\cC(z)$ has a singular expansion of order $3/2$ at $\rho$.
Since $\cG(z)=\exp(\cC(z))$ and $\exp$ is analytic everywhere (in particular at $\cC(\rho)$),
we conclude that $\cG(z)$ has also a singular expansion of order $3/2$ at $\rho$.
The transfer theorems of singularity analysis (see Section~\ref{sec:transfer}) yield an estimate of the form~\eqref{eq:estim_Gn},
where $\gamma=1/\rho$. See also~\cite{BePaSt09}.
\end{proof}
\subsection{Sufficient condition for subcriticality}
The following lemma gives a simple sufficient condition for subcriticality:
\begin{lem}\label{lem:criterion_sub}
Let $g(y)$ be a series with non-negative coefficients and positive radius of convergence $\eta$
such that $g'(y)\to\infty$ as $y\to \eta^-$. Let $f(z)$ be the unique solution of~\eqref{eq:fg}.
Then the pair $(f(z),g(y))$ is subcritical.
\end{lem}
\begin{proof}
Let $\rho$ be the radius of convergence of $f(z)$ and $\tau:=f(\rho)$.
From the definition of the subcriticality we need to show $\tau<\eta$.

Assume $\tau>\eta$. Then, by continuity of $f(z)$, there exists $0<z_0<\rho$
such that $f(z_0)=\eta$. Since $f(z)$ is regular at $z_0$ with positive derivative
and since $g(y)$ is singular at $\eta$, the function $g(f(z))$ must be singular at $z_0$.
Hence $z\ \!\exp(g(f(z)))$ must also be singular at $z_0$, in contradiction to the
fact that $f(z)$ is regular at $z_0$. Hence $\tau\leq \eta$.

Assume now that $\tau=\eta$. Differentiating~\eqref{eq:fg}, we obtain
$$
f'(z)=\frac{f(z)}{z}+g'(f(z))\ \!f'(z)f(z),
$$
which implies $f'(z)\geq g'(f(z))f'(z)f(z)$, for $z\in(0,\rho)$. Taking $y=f(z)$, it simplifies to
$g'(y)\leq 1/y$ for $y\in(0,\eta)$. This contradicts the fact that $g'(y)\to\infty$ as $y\to \eta$.
\end{proof}
Note that the pair $(f(z),g(y))$ is subcritical if $g(y)$ has a
square-root singular expansion at $\eta$. Therefore, a block-stable graph class $\cG$ is subcritical
if the EGF of $\cB'$ admits a square-root singular expansion.

\section{Subcritical classes: the unlabelled case}\label{sec:sub_unlabelled}
In this section $\tcG$ denotes a block-stable class  of unlabelled graphs and $\tcC$ (resp. $\tcB$) its subclass consisting of connected (resp. $2$-connected) unlabelled graphs in $\tcG$.

\subsection{Definition of subcriticality in the unlabelled case}
We have seen in Section~\ref{subsec:block_stable} that a block-stable class satisfies
$$
\tcG(z)=\exp\left(\sum\nolimits_{i\geq 1}\tfrac1{i}\tcC(z^i)\right),
$$
$$
\tcCp(z)=z\ \!\exp\left(\sum\nolimits_{i\geq 1}\tfrac1{i}Z_{\cB'}(\tcCp(z^i),\tcCp(z^{2i}),\tcCp(z^{3i}),\ldots)\right).
$$
The second equation can be rewritten as follows:
\beq\label{eq:fgA}
f(z)=z\ \!\exp\!\big(g(f(z),z)+A(z)\big),
\eeq
where
\begin{eqnarray}
f(z)&:=&\tcCp(z) \label{eq:def_f}\\
g(y,z)&:=&Z_{\cB'}(y,f(z^2),f(z^3),\ldots)\label{eq:def_g}\\
A(z)&:=&\sum\nolimits_{i\geq 2}\frac1{i}Z_{\cB'}(f(z^i),f(z^{2i}),\ldots)\label{eq:def_A}.
\end{eqnarray}
Note that, given a bivariate series $g(y,z)$ and a univariate series $A(z)$,
there is a unique series $f(z)$ that is a solution of~\eqref{eq:fgA} (because
the coefficients of $f(z)$ are determined uniquely iteratively) and $f(z)$ has
nonnegative coefficients if $g(y,z)$ and $A(z)$ have nonnegative coefficients.
\begin{definition}\label{def:sub_unl}
Let $g(y,z),A(z)$ be series with nonnegative coefficients, and let $f(z)$ be
the unique solution of~\eqref{eq:fgA} and $\rho$ the radius of
convergence of $f(z)$. Then the triple $(f(z),g(y,z),A(z))$ is called subcritical if
\begin{itemize}
\item[(i)] $\rho$ is non-zero,
\item[(ii)] $g(y,z)$ is analytic at $(f(\rho),\rho)$, and
\item[(iii)] the radius of convergence of $A(z)$ is larger than $\rho$.
\end{itemize}
An unlabelled block-stable graph class $\tcG$ is called subcritical if
\begin{itemize}
\item[(a)]
the triple $(f(z),g(y,z),A(z))$ with $f(z)$, $g(y,z)$, $A(z)$ defined as in~\eqref{eq:def_f}--\eqref{eq:def_A} is subcritical, and
\item[(b)]
the radius of convergence of the series $q(z):=Z_{\cC}(0,z^2,z^3,\ldots)$ is strictly larger than $\rho$.
\end{itemize}
\end{definition}
Note that for any block-stable graph class $\tcG$, the class $\tcCp$ of rooted connected
graphs from $\tcG$ dominates coefficient-wise the class of unlabelled rooted non-plane trees, whose coefficients grow exponentially; hence,
 $\rho\leq \rho^*\approx 0.33832$ (with $\rho^*$ the radius of convergence of
 unlabelled forests). Furthermore, we also have stability of
subcriticallity when we vary an additional variable $v$ locally around $1$
if the parameter that is counted by the exponent of $v$ has at most
linear worst case behaviour.

\subsection{Asymptotic estimate for a subcritical class}

\begin{lem}\label{lem:CpOGF}
Let $\tcG$ be an unlabelled subcritical block-stable graph class with $\tcC$ the connected subclass
and $\tcB$ the 2-connected subclass.
Let $\rho$ be the radius of convergence of $f(z):=\tcCp(z)$. Then $f(z)$
has a square-root singular expansion around $\rho$, and $(y,z)=(f(\rho),\rho)$ is a solution of the singular system
$$
y=z\ \!\exp(g(y,z)+A(z)),\ \ 1=y\ \!g_y(y,z),
$$
with $g(y,z)$ and $A(z)$ defined in~\eqref{eq:def_g} and \eqref{eq:def_A}.
Furthermore, if
$[z^n]\, \cCp(z)>0$ for $n\ge n_0$ then $\rho$ is the only singularity
on the circe $|z| = r$ and we $\cCp(z)$ can be continued analytically
to the region $D = \{ z\in \mathbb{C} : |z| < \rho+\epsilon,\
1 - z/\rho \not\in \mathbb{R}^-\}$ for some $\epsilon> 0$.
\end{lem}
\begin{proof}
Recall that the function $f(z)=\tcCp(z)$ is a solution of
$$
y=F(y;z)=z\exp(g(y,z)+A(z)).
$$
Note that $h(z)=Z_{\cB'}(f(z),f(z^2),\ldots)$
is bounded coefficient-wise above by $f(z)$ and hence the singularity of $h(z)$ is larger than $\rho$. Since $\rho\in(0,1)$
(by the remark just after Definition~\ref{def:sub_unl}),
the function $A(z)=\sum_{i\geq 2}h(z^i)/i$ is analytic at $\rho$.
By definition of subcriticality also, the function $g(y,z)$ is analytic at $(f(\rho),\rho)$.
Hence $F(y;z)$ is analytic at $(f(\rho),\rho)$. Since the system is clearly strongly recursive and
the function $f(z)$ aperiodic, we conclude from Theorem~\ref{Thsystem1} that $f(z)$ has a square-root expansion at $\rho$.
\end{proof}

\begin{lem}
Let $\tcG$ be an unlabelled subcritical block-stable graph class with $\tcC$ the connected subclass
and $\tcB$ the 2-connected subclass.
Let $\rho$ be the radius of convergence
of $f(z):=\tcCp(z)$. Define $R(s,z):=Z_{\cCp}(s,z^2,z^3,\ldots)$. Then $R(s,z)$
has a square-root singular expansion around $(\rho,\rho)$, and the singularity function $\xi(z)$ of $s\mapsto R(s,z)$ has a negative derivative at $\rho$.
\end{lem}
\begin{proof}
The bivariate series $R(s,z)$ is a refinement of $\tcCp(z)$, since $\tcCp(z)=R(z,z)$.
The equation~\eqref{eq:cyc_C_B} implies that
$y=F(y;z,s):=s\ \!\exp(g(y,z)+A(z))$,
with $g(y,z)$ and $A(z)$ defined in~\eqref{eq:def_g} and \eqref{eq:def_A}. The singular system for $R(s,z)$ is
$$y=s\ \!\exp(g(y,z)+A(z)),\ \ 1=y\ \!g_y(y,z).$$
This is the same as the singular system
of $f(z)$ (given in Lemma~\ref{lem:CpOGF}) except that the variable $z$ on the left-hand side of $\exp$ is replaced by the variable $s$.
By Lemma~\ref{lem:CpOGF},
$(y,z)=(f(\rho),\rho)$ is a solution of the singular system of $f(z)$, hence clearly
$(y;z,s)=(f(\rho);\rho,\rho)$ is a solution of the singular system of $R(s,z)$, and $F(y;z,s)$ is analytic at $(f(\rho);\rho,\rho)$, since $g(y,z)$ is analytic at $(f(\rho),\rho)$. Thus, Theorem~\ref{Thsystem1} ensures that $R(s,z)$ has
a square-root singular expansion at $(\rho,\rho)$.
In addition,  the singularity function $\xi(z)$ has a negative derivative,
since $F(y;z,s)$ depends only on $z$.
\end{proof}

\begin{thm}
Let $\tcG$ be an unlabelled subcritical block-stable graph class
such that $[z^n]\, \cCp(z)>0$ for $n\ge n_0$. Then
there exist constants
$c>0$ and $\gamma$ such that
$$
[z^n]\tcG(z)=c\ \!n^{-5/2}\ \!\gamma^n(1+o(1))\ \ \mathrm{as}\ n\to\infty
$$
for $\gamma\geq\gamma^*\approx 2.95576$, where $\gamma^*$ is the
growth rate of unlabelled forests.
\end{thm}
\begin{proof}
First we show that $\tcC(z)$ has a singular expansion of order $3/2$ at $\rho$.
Define $Q(s,z):=Z_{\cC}(s,z^2,z^3,\ldots)$ (note that $\cC(z)=Q(z,z)$).
The general relation $Z_{\cA'}=\frac{\partial}{\partial s_1}Z_{\cA}$
ensures that $R(s,z)=sQ_s(s,z)$, hence
$$
Q(s,z)=Q(0,z)+\int_0^sR(w,z)\frac{\mathrm{d}w}{w}.
$$
The term $Q(0,z)=Z_{\cC}(0,z^2,z^3,\ldots)=q(z)$ is analytic at $\rho$, by definition of subcriticality.
Since $R(s,z)$ has a square-root expansion at $(\rho,\rho)$, the integral term has a singular expansion of order $3/2$ at $(\rho,\rho)$ (see Section~\ref{sec:singularexpansion}) of the form
$$
Q(s,z)=a(s,z)+b(s,z)\cdot(1-s/\rho(z))^{3/2}.
$$
Therefore, $\cC(z)=Q(z,z)$ has a singular expansion of the form
$$
\cC(z)=a(z,z)+b(z,z)\cdot\big((\rho(z)-z)/\rho(z)\big)^{3/2}.
$$
Since $\xi(z)$ has a negative derivative at $\rho$ and $\xi(\rho)=\rho$, there exists a function $\lambda(z)$
analytic and nonzero at $\rho$ such that $\rho(z)-z=(1-z/\rho)\cdot\lambda(z)$. We conclude that $\cC(z)$
has a singular expansion of order $3/2$, of the form
$$
\cC(z)=\alpha(z)+\beta(z)\cdot\big(1-z/\rho\big)^{3/2},
$$
with $\alpha(z)=a(z,z)$ and $\beta(z)=b(z,z)\cdot\big(\lambda(z)/\rho(z)\big)^{3/2}$.

Recall that $\tcG(z)$ and $\tcC(z)$ are related by
$$
\tcG(z)=\exp(\tcC(z)+E(z)),\ \ \mathrm{with}\ E(z):=\sum\nolimits_{i\geq 2}\frac1{i}\tcC(z^i).
$$
Since $E(z)$ is analytic at $\rho$, the singular expansion of order $3/2$ at $\rho$
for $\tcC(z)$ yields also a singular expansion of order $3/2$ at $\rho$ for $\tcG(z)$.
The transfer theorems of singularity analysis then yield the estimate for $[z^n]\tcG(z)$.
\end{proof}

\subsection{Sufficient conditions for subcriticality}
Similarly as in the labelled case, we provide a list of conditions
that implies subcriticality, but will be convenient to check on examples (see Section~\ref{sec:ex} for the application):
\begin{lem}\label{lem:suf_crit_unl}
Let $\tcG$ be an unlabelled block-stable graph class with $\tcC$ and $\tcB$
the connected and 2-connected subclasses. Let $f(z), g(y,z)$ be defined as~\eqref{eq:def_f} and \eqref{eq:def_g} and let $\rho$ be the radius of convergence of $f(z)$.
For $z>0$ let $\eta(z)$ be the radius of convergence of $y\mapsto g(y,z)$.
Assume that
\begin{enumerate}
\item
there exist constants $c$ and $\gamma>0$ such that $[z^n]\tcCp\leq c\ \!\gamma^n$,
\item
the series $\ds g_y(y,z):=\frac{\partial}{\partial y}g(y,z)$ satisfies
$\ds\lim_{y\to\eta(\rho)^-}g_y(y,\rho)=+\infty$,
\item the function $\eta(z)$ is continuous at $\rho$, and
\SPP\SPP\SPP\SPP
\item the radius of convergence of $q(z)=Z_{\cC}(0,z^2,z^3\dots)$ is larger than $\rho$.
\end{enumerate}
Then the unlabelled class $\tcG$ is subcritical.
\end{lem}
\begin{proof}
We have to show that the list of four criteria above implies that
(i) $\rho$ is non-zero,
(ii) $A(z)$ is analytic at $\rho$, and
(iii) $g(y,z)$ is analytic at $(f(\rho),\rho)$.
The first criterion exactly implies (i). It is actually in $(0,1)$
(see the remark after Definition~\ref{def:sub_unl} about $\rho$ being smaller than $1$).
And we have already shown in Lemma~\ref{lem:CpOGF} that $\rho\in(0,1)$
automatically implies that $A(z)$ is analytic at $\rho$, which proves (ii).
Next we show (iii) holds.
First we show that $f(\rho)<\eta(\rho)$, that is, $y\mapsto g(y,\rho)$ is analytic at $f(\rho)$. If $\eta(\rho)<f(\rho)$, then $g(f(z),z)$ is infinite at $z=\rho$,
so $f(z)=z\ \!\exp(g(f(z),z)+A(z))$ is also infinite at $\rho$, which is impossible (any solution  of
a strongly recursive system is finite at its radius of convergence). The case $\eta(\rho)=f(\rho)$ is excluded
in a similar way as in Lemma~\ref{lem:criterion_sub}. More precisely, differentiating~\eqref{eq:fgA} gives
$$
f'(z)=f(z)/z+\Big(A'(z)+g_y(f(z),z)f'(z)+g_z(f(z),z)\Big)f(z).
$$
Hence, $f'(z)\geq g_y(f(z),z)f'(z)f(z)$ for $z\in(0,\rho)$, which yields $g_y(f(z),z)\leq 1/f(z)=O(1)$ as $z\to\rho$. This contradicts $g_y(y,\rho)\to\infty$ as $y\to \eta(\rho)$.
Thus, $\eta(\rho)\neq f(\rho)$ since the second criterion says that $g_y(y,z)$ is infinite at
$(\eta(\rho),\rho)$. So we have $f(\rho)<\eta(\rho)$, which ensures that $y\mapsto g(y,\rho)$ is analytic at $f(\rho)$.
But we need to prove a little stronger condition, namely that $g(y,z)$ is analytic at $(f(\rho),\rho)$. Due to the
continuity condition on $\eta(z)$, $f(z_0)<\eta(z)$ in a small interval around
$\rho$ and therefore $g(y,z)$ converges in a neighbourhood of $(f(\rho),\rho)$, i.e., $g(y,z)$ is analytic
at $(f(\rho),\rho)$.
\end{proof}

\section{Examples of subcritical graph classes}\label{sec:ex}
Recall that a block-stable graph class is completely determined by its 2-connected
subclass $\cB$ and that $\cB$ must contain the link graph $\ell$ (a graph with one edge together with its two labelled end vertices). Let $\cM=\cB\backslash \ell$.
We will deal with three block-stable classes and show subcriticality both in
the labelled and  unlabelled cases: the class of cacti graphs, where $\cM$ consists of (convex) polygons; the class of outerplanar graphs, where $\cM$ consists of dissections of (convex) polygons; and  the class of series-parallel graphs, where $\cM$ consists
of simple graphs obtained from a double
edge by repeatedly choosing an edge to be doubled or to have a vertex
inserted in its middle.
Using the subcriticality criteria introduced in Sections~\ref{sec:sub_labelled} and~\ref{sec:sub_unlabelled} we will show that these three block-stable classes are subcritical and therefore feature a universal asymptotic behaviour with subexponential term $n^{-5/2}$.
\begin{thm}\label{theo:ex}
The classes of cacti graphs, outerplanar graphs, and series-parallel graphs are subcritical
both in the labelled and unlabelled cases. As a consequence, the counting coefficient $g_n$
of each of
these classes -- $g_n=|\cG_n|/n!$ in labelled case, $g_n=|\tcG_n|$ in unlabelled case --
is asymptotically of the form
$$
g_n= g\ \!n^{-5/2}\rho^{-n}(1+o(1))
$$
for some constants $g>0$, $\rho\in(0,1)$. The first few digits of $\rho$ in labelled case and the approximate values of $\rho$ in unlabelled case are resumed in Table~\ref{table:constant-growth}.
\end{thm}

We do not claim here the originality of the above asymptotic estimates, except for unlabelled series-parallel graphs: labelled outerplanar and SP graphs are treated in~\cite{BoGiKaNo07}; unlabelled outerplanar graphs in~\cite{BoFuKaVi07b}; labelled cacti in~\cite{PaWe07} and unlabelled cacti in~\cite{Sh07}). What is novel, however, is how we have derived these asymptotic estimates, namely through a unified (firstly for labelled and unlabelled cases and secondly for various graph classes) method, based on the decomposition grammar, the singularity analysis and the subcriticality criteria that are easy to check, which we will show below. Analogous asymptotic estimates hold for classes of graphs that are stable under taking connected, 2-connected, \emph{and} 3-connected components, but have only finite 3-connected subclass (see Subsection~\ref{sub:3-connex}).

To prove Theorem~\ref{theo:ex}, we check whether the sufficient conditions for subcriticality
(Lemma~\ref{lem:criterion_sub} in the labelled case, Lemma~\ref{lem:suf_crit_unl} in the unlabelled case) are satisfied.
Throughout this section we use the functions and notations in Lemmas~\ref{lem:criterion_sub} and~\ref{lem:suf_crit_unl}.

\subsection{Cacti graphs.}
The asymptotic study of cacti graphs
is carried out in~\cite{PaWe07} for the labelled case and in ~\cite{Sh07} for the unlabelled case.
Cacti graphs are such that $\cM=\cB\backslash \ell$ consists of (unoriented convex) polygons. Therefore, the derived class $\cB'$ is
isomorphic to unoriented sequences of at least two vertices (because the polygon can be broken at the root-vertex).
Thus
$$
\cB'(y)=y+\frac{y^2}{2(1-y)},
$$
where the terms $y$ counts the link-graph with a pointed (i.e. distinguished but unlabelled) vertex.
The series $g(y)=\cB'(y)$ clearly diverges at its radius of convergence $1$. Therefore the class of
labelled cacti graphs is subcritical.

In the unlabelled case, we have to take automorphisms into account. The only possible symmetries
of unoriented sequences are the identity and the order-reversing of the sequence, therefore
$$
Z_{\cB'}(s_1,s_2,\ldots)=s_1+\frac{s_1^2}{2(1-s_1)}+\frac{1+s_1}{2(1-s_2)},
$$
and the series $g(y,z)=Z_{\cB'}(y,f(z^2),f(z^3),\ldots)$ satisfies the expression
$$
g(y,z)=y+\frac{y^2}{2(1-y)}+\frac{1+y}{2(1-f(z^2))}.
$$
From this we obtain the equation satisfied by $y=f(z)$:
$$
y=z\ \!\exp\left(\sum\nolimits_{i\geq 1}\tfrac1{i}h(z)\right),\ \ \mathrm{with}\ h(z)=f(z)+\frac{f(z)^2}{2(1-f(z))}+\frac{1+f(z)}{2(1-f(z^2))}.
$$
Since $f(\rho)<\infty$ and $h(z)\preceq f(z)$ (i.e. $h(z)$ is coefficient-wise dominated by $f(z)$), we have $h(\rho)<\infty$, so that $f(\rho)<1$.
As a consequence $1/(1-f(z^2))$ is analytic at $\rho$.
Looking at the expression of $g(y,z)$,
we see that the radius $\eta(z)$ of convergence of $y\mapsto g(y,z)$ satisfies $\eta(z)=1$ for $z$ close to $\rho$ (in particular $\eta(z)$ is continuous at $\rho$)
and $g_y(y,\rho)$ goes to infinity (double pole
in $y$) when $y\to 1$. It remains to check that $q(z)$ is analytic at $\rho$. The cycle index
sum of polygons is well-known (the automorphism group is the dihedral group for each polygon):
$$
Z_{\cB}(s_1,s_2,\ldots)=-\frac{s_1}{2}+\frac1{2}\sum_{r\geq 1}\frac{\phi(r)}{r}\log\left(\frac1{1-s_r}\right)+
\frac{s_1^2+2s_1s_2+s_2}{4(1-s_2)},
$$
where $\phi$ is the Euler totient function.
Hence
$$
q(z)=
\frac{f(z^2)}{4(1-f(z^2))}+\frac{1}{2}\sum_{r\geq 2}\frac{\phi(r)}{r}\log(1-f(z^r)).
$$
This function is clearly analytic at $\rho$, which concludes the proof that the
class of unlabelled cacti graphs is subcritical.

\subsection{Outerplanar graphs.}
Our second example, outerplanar graphs, has been studied asymptotically in~\cite{BoGiKaNo07}
for the labelled case
and in~\cite{BoFuKaVi07b} for the unlabelled case.
An \emph{outerplanar graph} is a graph that can be embedded in the plane
such that all vertices lie in the outer face. It is also defined as the class of graphs avoiding  $K_4$ and $K_{3,2}$ as minors. A well-known characterisation of 2-connected outerplanar graphs with at least $3$ vertices says that they are dissections of a (unoriented convex) polygon.
So the computation shares some resemblance with the one for cacti graphs, except that the polygon is filled
with chords in a planar way. In the labelled case, the classical dual construction says that dissections of \emph{oriented} polygon are in bijection with rooted plane trees
with no node of degree 2. The leaves of the tree correspond
to the edges (minus one) of the dissection, which are themselves
equinumerous with the (non-rooted) vertices of the dissection.
Therefore
$$
\cB'(y)=y+\frac1{2}T(y),\ \ \mathrm{with}\ T(y)=\frac{(T(y)+y)^2}{1-y-T(y)},
$$
where the first term $y$ in $\cB'(y)$ takes account of the derived link-graph.
The series $T(y)$ satisfies $T(y)=(1-3y-\sqrt{y^2-6y+1})/4$, therefore it has a square-root singularity
at its radius of convergence $\rho_T:=3-2\sqrt{2}$.
This ensures (from the remark after Lemma~\ref{lem:criterion_sub}) that
the class of labelled outerplanar graphs is subcritical.

Next consider the unlabelled case. We will first compute
$Z_{\cB'}$. The only possible symmetries
are the identity and the reflection along an axis passing by the rooted vertex.
One finds (see~\cite{BoFuKaVi07b,Vi05} for a detailed calculation)
$$
Z_{\cB'}(s_1,s_2,\ldots)=\frac1{2}T(s_1)+\frac{s_1+s_2}{2s_2^2}T(s_2).
$$
Hence the series $g(y,z)=Z_{\cB'}(y,f(z^2),f(z^3),\ldots)$ satisfies
$$
g(y,z)=\frac1{2}T(y)+\frac{y+f(z^2)}{2f(z^2)^2}T(f(z^2)).
$$

We first observe that $g(f(\rho,\rho)$ is finite, since it is smaller than $f(\rho)$
which is finite, more precisely, $g(f(\rho),\rho)$ gathers from $f(\rho)$ the connected
rooted graphs with a unique block incident to the pointed vertex.
In particular $T(f(\rho))$ is finite, which ensures that $f(\rho)\leq\rho_T$.
Since $\rho<1$, $f(z^2)$ is strictly smaller than $\rho_T$ in a neighbourhood of $\rho$,
so
$T(f(z^2))$ is analytic at $\rho$. Consequently, for any $z$ close to $\rho$
the dominant singularity $\eta(z)$ of  $y\mapsto
g(y,z)$ is $\rho_T$ (because of the first term $T(y)/2$), hence $\eta(z)$ is
continuous at $\rho$ (it is actually constant equal to $\rho_T$
around $\rho$), so the criterion (2)  of subcriticality
is satisfied.
Moreover,  $g(y,z)$ inherits from $T(y)$ a square-root expansion at $\rho_T$, so $g_y(y,z)$ diverges as $y\to\rho_T$, thus the criterion (3) is satisfied.

It remains to check the criterion (4). To find an expression for $Z_{\cB}$
one has to enumerate dissections of a polygon under all possible symmetries
(rotation and reflection), where the duality with trees helps to formulate
the decompositions. All calculations were done (see~\cite{BoFuKaVi07b,Vi05}):
$$
Z_{\cB}(s_1,\dots)=\frac{s_1}{4}T(s_1)-\frac1{2}\sum_{d\geq 1}\log(1-s_d-T(s_d))+\frac{s_1^2+s_2^2+2s_1s_2}{4s_2^2}T(s_2)+P(s_1,s_2),
$$
where $P(s_1,s_2)$ is a certain polynomial in $s_1$ and $s_2$. Consequently,
 the series $q(z)=Z_{\cB}(0,f(z^2),f(z^3),\ldots)$ satisfies
$$
q(z)=-\frac1{2}\sum_{d\geq 2}\log(1-f(z^d)-T(f(z^d)))+\frac{T(f(z^2))}{4}+P(0,f(z^2)).
$$
The argument of $\log$ is non-zero around $\rho$,
because for $d\geq 2$, $f(z^d)<\rho_T$ in a neighbourhood of $\rho$,
and  $T(y)+y$ is less than $1$ at its singularity $\rho_T$
(precisely it equals $(5\rho_T+1)/4<1$).
Hence the criterion (4) is satisfied. In conclusion, Lemma~\ref{lem:suf_crit_unl} implies  that the class of unlabelled outerplanar graphs
are subcritical class.

\subsection{Series-parallel graphs.}\label{sec_sp_unl}
We use the decomposition grammar for 2-connected SP-graphs developed in~\cite{ChFuKaSh07} and~\cite{GLLW08}.
More precisely, we will use the decomposition grammar for the so-called series-parallel \emph{networks}. A network is a connected SP graph with two pointed (distinguished but unlabelled) vertices, which are called the poles and denoted by $-$ and $+$,
such that adding the edge $(-,+)$
results in a 2-connected SP graph (possibly the two vertices $-$ and $+$ are already adjacent in the network).
We call series-networks those that are not 2-connected (when deleting the edge between the poles if any) and parallel networks
those that are 2-connected with at least 2 edges. The classes of networks, series-networks and parallel-networks
are denoted respectively $\cD$, $\cS$, and $\cP$. The link-graph consisting of one edge from $-$ to $+$ is denoted by $e$ to distinguish it from the link-graph $\ell$ consisting of an edge and two labelled end vertices.

In the labelled case we obtain the decomposition grammar and its corresponding system of EGFs (on the right is the associated system for the EGFs, $\Set_{\geq k}$ means Set constrained to have at least $k$ components):
\vspace{0.2cm}

$\ds
\left\{
\begin{array}{rcl}
\cD&=&e+\cS+\cP,\\
\cS&=&\cD\cdot\cZ\cdot(e+\cP),\\
\cP&=&e\cdot\Set_{\geq 1}(\cS)+\Set_{\geq 2}(\cS)
\end{array}
\right.
$
$
\Rightarrow
$
$
\left\{
\begin{array}{rcl}
\cD(y)&=&1+\cS(y)+\cP(y),\\
\cS(y)&=&y\ \!\cD(y)(1+\cP(y)),\\
\cP(y)&=&2\exp(\cS(y))-\cS(y)-2
\end{array}
\right.
$
\vspace{0.2cm}

The system for the EGFs is clearly strongly recursive and the series $\cD(y)$ is easily aperiodic. In addition,
the function-system $\mathbf{F}(y_1,y_2,y_3,y)$ with $y_1=\cD(y)$, $y_2=\cS(y)$, $y_3=\cP(y)$ defined by the right-hand side of the system is clearly analytic
everywhere (because $\exp$ is analytic everywhere). In addition, easy lower and upper bounds imply
that the radius of convergence $\eta_{\cD}$ of $\cD(y)$ is in $(0,1)$.
Hence Theorem~\ref{Thsystem1} applies,
ensuring that all the series for networks have a square-root expansion at their common radius of convergence $\eta_{\cD}$.
Now we want to show that $g(y):=\cB'(y)$ has radius of convergence $\eta_{\cD}$ and that $g'(y)$ goes to $\infty$
as $y\to \eta_{\cD}^-$. First, note that in the labelled framework the number of networks is at least as large as the
number of vertex-rooted 2-connected SP graphs (because each graph in $\cB'$ gives rise to $d\geq 1$ networks, with
$d$ the degree of the root-vertex). So the series $y^2\cD(y)$ dominates coefficient-wise the series $y\cB'(y)$, shortly
written $\cB'(y)\preceq u\cD(y)$. Surprisingly the domination also goes the other way. Precisely speaking, we have the inclusion $\cZ^3\cdot\cD\subseteq\cB'$, as illustrated in Figure~\ref{fig:netToVert}, so that $y^3\cD(y)\preceq\cB'(y)$.
\begin{figure}[htb]
\begin{center}
\includegraphics[width=6.4cm]{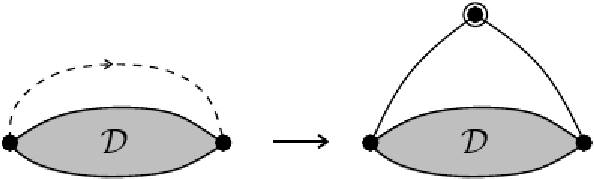}
\end{center}
\caption{Injection from networks to vertex-rooted 2-connected SP graphs
by adding one vertex and two edges.}
\label{fig:netToVert}
\end{figure}

We thus have $y^3\cD(y)\preceq \cB'(y)\preceq u\cD(y)$, which ensures that $\cB'(y)$ has the radius of convergence $\eta_{\cD}$
and that $g(y)=\cB'(y)$ satisfies $g'(y)\to\infty$ as $y\to \eta_{\cD}^-$, since $\cD(y)$ has a square-root expansion and hence $\cD'(y)\to\infty$ as $y\to \eta_{\cD}^-$). This concludes  that the class of labelled SP graphs is subcritical.

Now we turn to the unlabelled case, which is technically more involved.
For each class $\cN\in\{\cD,\cS,\cP\}$ of networks,  the automorphisms have to fix each of the two poles $+$ and $-$,
and the cycle index sum $Z_{\cN}(s_1,s_2,s_3,\ldots)$ is defined as a sum of monomials over all automorphisms, as in Section~\ref{sec:def_series}. We define $N=N(y,z):=Z_{\cN}(y,f(z^2),f(z^3),\ldots)$.
For networks, the automorphisms have to fix each of the two poles $+$ and $-$. However, it is also useful
to consider networks up to exchanging the poles and their corresponding series $\oZ_{\cN}(s_1,s_2,s_3,\ldots)$
defined similarly as the cycle
index sum, but summing over the automorphisms exchanging $+$ and $-$ (instead of those fixing $+$ and $-$ as in
the series $Z_{\cN}$). Define the bivariate series $\oN=\oN(y,z):=\oZ_{\cN}(y,f(z^2),f(z^3),\ldots)$ for each class $\cN\in\{\cD,\cS,\cP\}$ of networks. For a bivariate series $f(y,z)$ and $k\geq 1$ define $f^{(k)}(y,z):=f(y^k,z^k)$. The decomposition for pole-fixing networks and pole-exchanging networks give rise to explicit systems for the cycle index sums. Under the specialization $(s_1=y,s_2=f(z^2),s_3=f(z^3),\ldots)$, these yield the two systems (where the argument $(y,z)$ of the series is omitted):
$${\mathrm{(PF)}}\ds
\left\{
\begin{array}{rcl}
\SPP\SPP\SPP
\ds D&=&1+S+P,\\
\ds S&=&y\cdot D\cdot (1+P),\\
\ds P&=&2\exp\left(\ds\sum\nolimits_{i\geq 1}\frac1{i}S^{(i)}\right)-S-2,
\end{array}
\right.
$$
and for the pole-exchanging networks
$$
{\mathrm{(PE)}}\ds
\left\{
\begin{array}{rcl}
\SPP\SPP\SPP\SPP\SPP\SPP
\ds \oD&=&1+\oS+\oP,\\
\ds \oS&=&D^{(2)}\cdot (y+f(z^2)(1+\oP)),\\
\ds \oP&=&2\exp\left(\ds\sum\nolimits_{k\geq 1}\frac{S^{(2k)}}{2k}+\frac{\oS^{(2k-1)}}{2k-1}\right)-\oS-2.
\end{array}
\right.
$$
The series $g(y,z):=Z_{\cB'}(y,f(z^2),f(z^3),\ldots)$ can be
expressed in terms of the series in the above two systems (this is done in~\cite{GLLW08}
using the dissymmetry theorem):
\begin{eqnarray*}
g(y,z) & = & \frac1{2}y^2(1+P) D+\frac1{2}f(z^2) D^{(2)}\cdot\left(1+y+y P^{(2)}+\oP\right)+\\
&&2y\exp\left(\sum\nolimits_{i\geq 1}\frac1{i}S^{(i)}\right)-2y-2yS-\frac1{2}yS^2-\frac1{2}yS^{(2)}-ySP.
\end{eqnarray*}
As in Lemma~\ref{lem:suf_crit_unl} we let $f(z):=\tcCp(z)$, $g(y,z):=Z_{\cB'}(y,f(z^2),f(z^3),\ldots)$, and  let $\rho$ be the radius of convergence of $f(z)$ and  $\tau=f(\rho)$. In addition, we let $\eta_D(z)$ and $\eta_{\oD}(z)$ be the singularity functions of $y\mapsto D(y,z)$ and $y\mapsto \oD(y,z)$.

\begin{lem}
Each of the series $D(s,z),S(s,z),P(s,z)$ admits a square-root expansion at $(\eta_D(\rho),\rho)$.
\end{lem}
\begin{proof}
As $Z_{\cB'}(f(z),f(z^2),\ldots)$ counts derived connected SP graphs with a unique
block incident to the pointed vertex and $f(z)$ counts all rooted SP graphs, we have $g(f(\rho),\rho)=Z_{\cB'}(f(\rho),f(\rho^2),f(\rho^3),\ldots)\leq f(\rho)$.
In addition $f(\rho)<\infty$ since $f(z)$ is a solution of nonlinear functional equations.
Hence $g(f(\rho),\rho)<\infty$.
As $\cZ^3\cdot\cD\subseteq\cB'$ (see Figure~\ref{fig:netToVert}) we have
$$u^3\ \!D(y,z)\preceq g(y,z).$$
Hence, by domination, $D(f(\rho),\rho)<\infty$ and therefore  the singularity of $y\mapsto D(y,\rho)$ is larger than $f(\rho)$,
which implies in particular that $z=\rho$ is an admissible valuation.
The equation-system (PF) $\mathbf{F}(\mathbf{y};z)$ for $\mathbf{y}=(D,S,P,y)$ satisfies all conditions required in Theorem~\ref{Thsystem1}
(aperiodicity of $D$, strongly recursive system, analyticity of the functional system $\mathbf{F}(\mathbf{y};z)$ at the required
point). Therefore the series $D$, $S$, and $P$ admit singular expansions at $(\eta_D(\rho),\rho)$.
\end{proof}

\begin{lem}\label{lem:square_D}
If $\eta_D(\rho)<\eta_{\oD}(\rho)$ then
the series $\oD(s,z),\oS(s,z),\oP(s,z)$ are analytic at $(\eta_D(\rho),\rho)$, if not~\footnote{This case
is very unlikely to happen, but discarding it would require some numeric computation of the functions $\eta_D$
and $\eta_{\oD}$.}, then these series admit a square-root
expansion at $(\eta_{\oD}(\rho),\rho)$.
\end{lem}
\begin{proof}
Note that Figure~\ref{fig:netToVert} yields an injection from the
pole-exchanging automorphisms of $\cD$ to the automorphisms of $\cB'$.
As a consequence, $s_2\oZ_{\cD}\preceq Z_{\cB'}$
so that
$$z^2\oD(y,y)\preceq g(y,z).$$
Hence, $Z_{\cB'}(f(\rho),f(\rho^2),f(\rho^3),\ldots)=\tau<\infty$
implies that $\oD(f(\rho),\rho)<\infty$. Therefore $\eta_{\oD}(\rho)\geq f(\rho)$
so that $z=\rho$ is an admissible valuation for $\oD(y,z)$.
Note that the functional system  (PE) $\mathbf{G}(\mathbf{y};z)$ with $\mathbf{y}=(\oD,\oS,\oP,y)$ is clearly analytic at a given point $(\oD,\oS,\oP,y;z)$ (where $\oD$, $\oS$, $\oP$
are seen as independent variables) if and only if $D$ is analytic at $(y^2,z^2)$
and $f$ is analytic at $z^2$. Since $\rho<1$ and $\eta_{\cD}(z)$ is decreasing in $z$,
$\mathbf{G}$ is analytic at any point such that $|y|\leq \eta_{D}(\rho)$
and $|z|\leq\rho$.
Therefore, the only cause of singularity for $\oD$ in this domain is a branch point (i.e.,
a solution of the singularity system). From Theorem~\ref{Thsystem1}, we conclude that in such a situation, $\oD$, $\oS$, and $\oP$
 have a square-root expansion at $(\eta_{\oD}(\rho),\rho)$.
\end{proof}
\begin{lem}\label{lem:final_SP}
The class of unlabelled series-parallel graphs is subcritical.
\end{lem}
\begin{proof}
We need to check that all four criteria in Lemma~\ref{lem:suf_crit_unl} are satisfied.
Easy upper and lower bounds imply that $\rho\in(0,1)$, so the criterion (1) is satisfied.
Let $\eta(z)$ the singularity function of $y\mapsto g(y,z)$.
Expression for $g(y,z)$ in terms of the series of networks (pole-fixing and pole-exchanging),
ensures that $\eta(z)=\mathrm{min}\{\eta_{D}(z),\eta_{\oD}(z)\}$ for $z>0$.
Hence $\eta(z)$, as the minimum of two continuous functions, is also continuous at $\rho$, so the criterion (3)  is verified. If $\eta(\rho)=\eta_D(\rho)$ (the most likely case), then $\lim_{y\to\eta(\rho)^-}g_y(y,\rho)=\infty$ because $y^3D(y,z)\preceq g(y,z)$ and because the square-root expansion of $D(y,z)$ at $(\eta(\rho),\rho)$ yields $\lim D_y(y,\rho)=\infty$ as $y\to\eta(\rho)$. If $\eta(\rho)=\eta_{\oD}(\rho)$, then applying Lemma~\ref{lem:square_D} $\oD(y,z)$ has a square-root expansion at $(\eta(\rho),\rho)$. Again we have $\lim g_y(y,\rho)=\infty$ as $y\to\eta(\rho)$ because $z^2\oD(y,z)\preceq g(y,z)$ and because the square-root expansion of $\oD(y,z)$ at $(\eta(\rho),\rho)$ yields $\lim\oD_y(y,\rho)=\infty$ as $y\to\eta(\rho)$. Thus the criterion (2) is verified.

It remains to check the criterion (4).
From the work in~\cite{ChFuKaSh07,GLLW08} one can extract an expression for $K(y,z):=Z_{\cB}(y,f(z^2),f(z^3),\ldots)$
(longer than the one of $g(y,z)$, but of the same aspect) in terms of the series $\{D,S,P,\oD,\oS,\oP\}$.
This ensures that the singularity function $\eta_H(z)$ of $y\mapsto H(y,z)$ satisfies
$\eta_H(z)=\mathrm{min}\{\eta_D(z),\eta_{\oD}(z)\}=\eta(z)$.
Since $\eta_H(\rho)\geq f(\rho)$ and since $\eta_H$ is continuous at $\rho$,
we have $\eta_H(z)>0$ for $z$ around $\rho$. Hence $H(0,z)$ is analytic at $\rho$.
In addition one can show that $Z_{\cC}(0,z^2,z^3,\ldots)=Z_{\cB}(0,f(z^2),f(z^3),\ldots)$, see~\cite{GLLW08} (a combinatorial interpretation is that if there is no fixed vertex
in an automorphism $\sigma$ of a connected graph, then there is a unique 2-connected block $B$ fixed by
$\sigma$, and the
automorphism induced on $B$ has no fixed vertex as well).
Hence $q(z)=Q(0,z)=H(0,z)$,
 which ensures that $q(z)$ is analytic at $\rho$.
\end{proof}

\subsection{Graphs defined by their $3$-connected components.}\label{sub:3-connex}

 The same
arguments used in the enumeration of the unlabelled SP-graphs can be
adapted to get similar results for families of graphs which are
defined by a \emph{finite} set of $3$-connected components
(see~\cite{GiNoRu09} for a proper definition of these families). In
these families, the strategy used is the same as in SP-graphs: we
study the enumeration of networks in order to get the asymptotic
counting of the graphs with a lower level of connectivity. The main
difference is that there is an additional line in the equation-system
for networks, which involves the \emph{Walsh polynomial}
(roughly speaking, the cycle index sum associated to the finite
family of $3$-connected components; see~\cite{Gagarin200751, GLLW08}
 for a
detailed definition). The Walsh polynomial has two kinds
of variables: $e_1,\ldots,e_r$ and $\overline{e_1},\ldots,\overline{e_k}$.
In the equation-system for networks, these variables are to be substituted
by $D^{(1)},\ldots,D^{(r)}$ and by $\overline{D}^{(1)},\ldots,\overline{D}^{(k)}$,
respectively (in SP-graphs this polynomial vanishes,
because SP-graphs do not have $3$-connected components). Since
the Walsh polynomial is an entire function in all its variables and has
positive coefficients, the system remains positive and the
argument of Lemmas~\ref{lem:square_D}, \ref{lem:final_SP}
can be
adapted to this general setting in order to assure a square-root
singularity for the counting series associated to networks, and the rest of the results can be easily adapted to this general framework.

\section{Limit laws}\label{sec:limitlaws}

In this section we study the limit distribution of a graph parameter
defined on a uniformly distributed random graph on $n$ vertices.
More concretely, let $X_n$ be the associated random variable. We
show that all the parameters under study invariably converges in distribution to a normal
limit distribution in the form
\[
\frac{X_n - \mathbb{E}\, X_n}{\sqrt{\mathbb{V}{\rm{ar}}\, X_n}} \to N(0,1),
\]
where $\mathbb{E}\, X_n = \mu n + O(1) \quad\mbox{and}\quad
\mathbb{V}{\rm{ar}}\, X_n = \sigma^2 n + O(1)$ for computable constants $\mu,\sigma^2>0$. The main results used
in this section are Theorem $2.23$ of~\cite{Drmota} (adapted
in this paper as Theorem~\ref{ThcombCLT1}).
In order to decide whether $\sigma^2> 0$ we can either calculate it directly
or use Lemma~\ref{Lepositive}. For example, we can check the conditions
of Lemma~\ref{Lepositive} with the help of small values of $n,m,k$.
This applies for all kinds of graphs and parameters we consider here.
However, in any case (if $\sigma^2 = 0$ or $\sigma^2 > 0$)
these parameters are concentrated around its expected value.

The parameters discussed in this section are the
number of edges, the number of blocks and the number of cut-vertices. All these parameters can be encoded in the generating
function framework of the form
$$\cG(z,v)=\exp(\cC(z,v)),\,\,\cCp(z,v)=F(\cCp(z,v);z,v)$$
for the labelled case, and
$$\tcG(z,v)=\exp\left(\sum\nolimits_{k\ge 1}\tcC(z^k,v^k)/k\right),\,\,\tcCp(z,v)=F(\tcCp(z,v);z,v)$$
for the unlabelled case. In both cases, $z$ marks vertices and $v$
marks the parameter under study. The singularity curve $z=\rho(v)$ for the GF associated to general graphs is the same as the
one for connected graphs, because both functions
$\exp$ and $\expb$ are analytic functions. As a consequence, the same
limit law holds for both general and connected graphs (hence we can restrict our study to connected graphs). In order to study these limit laws we
analyse the characteristic system
\begin{equation}\label{eq:system-parameter}
y=F(y;z,v),\,\,\,1=\frac{\partial F}{\partial y}(y;z,v).
\end{equation}
Observe that we recover the univariate system
$y=F(y;z),\,\,\,1=\partial F/\partial y(y;z)$ by
evaluating~\eqref{eq:system-parameter} at $v=1$. For the classes
under study we have shown that this univariate system of equations
is analytically well-founded (in the sense of
Definition~\ref{Def3}), hence by Theorem~\ref{Thsystem1}, $y(z,1)$
is the unique solution and has a square-root expansion around its
smallest real singularity. This expansion can be extended to a real
neighbourhood of $v=1$ with expression $y(z,v)$.

In this section we obtain the exact parameters for the labelled
case. Unfortunately in the unlabelled framework there are no explicit formulas for the expectation and the variance, hence we
content ourselves with the normal condition derived from
Theorem~\ref{ThcombCLT1} combined with Lemma~\ref{Lepositive}.
Finally, local limit laws are also discussed for these parameters.

\subsection{The labelled case}

The labelled case have been partially treated in~\cite{GiNoRu09}, where the authors developed a case-by-case analysis.
In this section we rediscover these results using a different and more general argument. Recall that
$\tau=\cCp(\rho)$ is the unique solution of the equation $y \cB''(y)=1$ on the
region where $\cB(y)$ is analytic, and $\rho$ is the radius of
convergence of $\cC(z)$.

We use a secondary variable $v$ to mark edges. For a graph class
$\cG$ we write $\cG=\cup_{n,m}\cG_{n,m}$, where $n$ denotes the
number of vertices and $m$ denotes the number of edges. We define
the bivariate generating function as
$$
\cG(z,v)=\sum_{n,m}\frac1{n!}|\cG_{n,m}|z^n v^m.
$$
The series is exponential according to the (labelled) vertices
and ordinary according to the (unlabelled) edges. A similar
definition can be stated for both connected and $2$-connected graphs
in $\cG$ ($\cC(z,v)$ and $\cB(z,v)$ respectively). For a
block-stable graph class $\cG$, the block-decomposition translates
to
$$\cCp(z,v)=z\ \!\exp(\cB'(\cCp(z,v),v)),$$
where $\cB'(z,v)$ denotes the derivative of $\cB(z,v)$ with respect to the first variable.
Observe that we need to deal with the bivariate
GF associated to 2-connected graphs of the class. All subcritical
classes under study are subclasses of planar graphs, hence the
number of edges is linear. Consequently the subcritical condition is stable under a slight
variation of $v$ around $v=1$. We denote by $\cB_v'(z,v)=\frac{\partial}{\partial v}\cB'(z,v)$, the
derivative of $\cB'(z,v)$ with respect to the second variable.

The next parameter considered is the number of blocks, which is coded by the following equation
\begin{equation*}
\cCp(z,u) = z \ \!\exp\left(u \cB'(\cCp(z,u))\right),
\end{equation*}
where the secondary variable $u$ marks the number of blocks. Finally, for the number of  separating
vertices, also called \emph{cut-vertices}, we consider the bivariate GF
\begin{equation*}
\cCp(z,w)=z\ \!w\
\!(\exp\left(\cB'\left(\cCp(z,w)\right)\right)-1)+z.
\end{equation*}
Observe however that $w$ does not mark \emph{exactly} the
number of separating vertices, as the root-vertex is always considered
as a cut-vertex. However, the limit law for the number of cut-vertices does not depend on the behaviour of a single vertex.

All the results are contained in Table~\ref{tab:constants}. Despite the methodology used here (direct application of Theorem~\ref{ThcombCLT1}) is different from the one of~\cite{GiNoRu09},  the parameters we
obtain are the same.
{\renewcommand{\arraystretch}{1.2}
\begin{table}[htb]
$$
\begin{tabular}{|c|c|c|c|}
  \hline
  Parameter & $F(y;z,v)$ & $\mu$ & $\sigma^2$ \\ \hline\hline
   Edges & $z\exp(\cB'(y,v))$ & $\cB_v'(\tau,1)$ & $1-\frac{(\tau\cB_{v}''(\tau,1))^2}{1+\tau^2
\cB'''(\tau)}$ \\
     Blocks & $z\exp(v\cB'(y))$ & $\log\left(\frac{\tau}{\rho}\right )$ & $\log\left(\frac{\tau}{\rho}\right)-{1 \over 1+\tau^2 \cB'''(\tau)}$ \\
    Cut-vertices & $zv\left(\exp(\cB'(y))-1\right)+z$ & $\left(1- {\rho \over \tau}\right)$ & $\left(\frac{\rho}{\tau}\right)^2\left(\frac{\tau}{\rho}-1-\frac{1}{1+\tau^2\cB'''(\tau)}\right)$ \\
  \hline
\end{tabular}
$$\bigskip
\caption{Parameters with the corresponding value of $\mu$ and
$\sigma^2$.} \label{tab:constants}
\end{table}
}

\subsection{The unlabelled case}
The same arguments used in the labelled framework work for GFs in the unlabelled framework, and Theorem~\ref{ThcombCLT1} assures asymptotic normal limit
distributions (with a non-vanishing variance) for these parameters. However, in this setting, expressions for the derivatives of $F$ are much involved, and consequently expressions for $\mu$ and $\sigma$ are complex. We cannot obtain \emph{closed} expressions (as the ones in
Table~\ref{tab:constants}) for both the expectation and the variance of the resulting random variables.

Using the techniques introduced in Section~\ref{sect: constant}
one can find numerical approximations  for the derivatives of the function $F(y;z,v)$. Such computations are conducted in~\cite{BoFuKaVi07b} for the number of edges in random unlabelled outerplanar graphs.

\subsection{Local limit laws}
In addition to the central limit laws described above, (much stronger) local limit laws hold for the parameters under study, either in
the labelled or the unlabelled framework. The main results used here are~\eqref{eqlocalCLT} and the
asymptotic estimates for both $[z^n]\cCp(z)$ and $[z^n]\tcCp(z)$. For conciseness, we state the result
for \emph{connected} graphs, but a similar result is also valid for
\emph{general} graphs.

\begin{thm}Let $\cG$ be a subcritical graph class
with the property that $[z^n] \cC(z) > 0$ for $n\ge n_0$.
 Let $X_n$ be a random variable for a graph parameter studied above (the number of edges, blocks, or cut-vertices) defined on a random uniformly
distributed connected graph of $\cG$ on $n$ vertices and
assume that $\sigma^2> 0$.
Let $m=\mu n+O(n^{1/2})$.
Then
\begin{equation*}
\mathbb{P}\left(X_n=m\right)
= \frac 1{\sqrt{2\pi  n}\sigma} e^{-(m-\mu n)^2/(2 \sigma^2
n)}\left(1+O(n^{-1/2})\right).
\end{equation*}
\end{thm}

\begin{proof} In all cases under study assumptions stated in the remark following the proof of Lemma~\ref{Lepositive} holds. The result is obtained by dividing the estimate in~\eqref{eqlocalCLT} by either $[z^n]\cCp(z)$ or $[z^n]\tcCp(z)$.
\end{proof}

\section{Degree distribution}\label{sec:degdis}
In this section we discuss the degree distribution of graphs of a subcritical graph family. We denote by $X_n^k$ and $\widetilde{X}_n^k$ the random variables which count the number of vertices of degree $k$ in a randomly chosen member of size $n$ of the family $\mc{G}$ and $\widetilde{\mc{G}}$, respectively. We further denote by $d_k$ and $\widetilde{d}_k$ the probability that the root vertex of a randomly chosen member of $\mc{G}'$ or $\widetilde{\mc{G}}'$ is $k$. 
The main tool to obtain asymptotic results is the analogue on systems of equations of Theorem \ref{ThcombCLT1}.

\subsection{The labelled case}\label{subsec:lab}
Consider the series
 \[
B_j'(z,\mathbf{v}) =
\sum_{n,n_1,\ldots,n_k,n_\infty}
b_{j;n;n_1,\ldots,n_k,n_\infty}' v_1^{n_1}\cdots v_k^{n_k}v_\infty^{n_\infty}\frac{z^n }
{n!},
\]
where we use the notation $\mathbf{v}=(v_1,\ldots,v_k,v_\infty)$ and $b_{j;n;n_1,\ldots,n_k,n_\infty}'$ is the number of derived
2-connected graphs with $1+n =1+n_1+\cdots+n_k+n_\infty$ vertices such that one vertex of degree
$j$ is marked and the remaining $n$ vertices are labelled by
$1,2,\ldots,n$ and $n_\ell$ vertices have degree $\ell$,
$1\le \ell \le k$, and $n_\infty$ vertices have degree greater
than $k$.
By definition we have
\[
B'(z) = \sum_{2\le j\le \infty} B_j'(z,\mathbf{1}).
\]
Hence the radius of convergence of the functions $B_j'(z,\mathbf{1})$
is greater or equal than the radius of convergence of $B'(z)$.

We introduce an analogous series
\[
C_j'(z,\mathbf{v})
= \sum_{n,n_1,\ldots,n_k,n_\infty}
c_{j;n;n_1,\ldots,n_k,n_\infty}'
v_1^{n_1}\cdots v_k^{n_k}v_\infty^{n_\infty}\frac{z^n }
{n!},
\]
for  derived connected graphs.
We further set
\[B'(z,w) = \sum_j B_j'(z,\mathbf{1})w^j = \sum_{n,j}b'_{n,j}z^nw^j,\]
and $b'_{n,j}$ is the number of derived $2$-connected graphs with $n+1$ vertices such that one vertex of degree $j$ is marked and the remaining $n$ vertices are labelled by $1,2,\ldots,n$. Analogously, we define the function $C'(z,w)=\sum_jC'_j(z,\mathbf{1})w^j$. According to the block decomposition of connected graphs
\begin{equation*}\label{eq.blockdec}
 C'(z,w)=\exp(B'(zC'(z),w)).
\end{equation*}
In the following, we set $B_j'(z):=B_j'(z,\mathbf{1})$. We will also use the series $B(z,\mathbf{v})$, which is the refined version of the generating function $B(z)$ of unrooted blocks taking into account all vertex degrees.
\begin{thm}\label{thm:deg}
Let $\mathcal{G}$ be a family of random subcritical graphs
(such that $[z^n]\,\cC(z)> 0$ for $n\ge n_0$)
and $\mathcal{G}'$ be its derived family. Then for $k$ fixed,
\begin{enumerate}
 \item\label{1} the limiting probability $d_k$ that the pointed vertex of a member of $\mathcal{G}'$ has degree $k$ exists and is given by
\begin{equation*}
d_k=\rho \left( \sum_{i=1}^k  B_i''\left(\rho C'(\rho)\right)C_{k-i}'(\rho)\right).
\end{equation*}
\item\label{2} the random variable $X_n^k$ that counts the number of vertices of degree $k$ in a randomly chosen member of $\mathcal{G}$ satisfies a central limit theorem with mean $\mathbb{E}\, X_n^k = d_k n + O(1)$ and variance $\mathbb{V}{\rm{ar}}\, X_n^k= \sigma_k^2 n + O(1)$, where $\sigma_k$ is a constant depending on $k$.
\end{enumerate}
\end{thm}
\begin{proof}[Proof of Theorem~\ref{thm:deg}~\eqref{1}]
We use the following lemma, which makes it easy to compute $d_k$. The proof is the same as the one for \cite[Theorem 9.17]{Drmota}.
\begin{lem}\label{lem:p(w)lab}
 The generating function $p(w) = \sum d_kw^k$ satisfies
\begin{equation*}
p(w) = \rho e^{B'(z,w)}\frac{\partial}{\partial z} B'(z,w) \bigg|_{z=\rho C'(\rho)},
\end{equation*}
and $p(1)=1$, thus the $d_k$'s define indeed a probability distribution.
\end{lem}
\noindent
Therefore, we have
\begin{align*}
d_k&=[w^k]p(w)
= [w^k]\rho e^{B'(z,w)}\frac{\partial}{\partial z} B'(z,w)\bigg|_{z=\rho C'(\rho)}\\
&=\rho\left(\sum_{i=0}^k[w^{k-i}]e^{B'(z,w)}[w^{i}]\frac{\partial}{\partial z} B'(z,w)\right)\bigg|_{z=\rho C'(\rho)}\\
&=\left(\sum_{i=1}^k\frac{\partial}{\partial z} B'_i(z)[w^{k-i}]\left(1+\sum_kB_k'(z)w^k + \frac{(\sum_kB_k'(z)w^k)^2}{2}+\cdots\right)\right)\bigg|_{z=\rho C'(\rho)}\\
&=\left(\sum_{i=1}^k\frac{\partial}{\partial z} B'_i(z)\left(\sum_{m=1}^{k-i}\sum_{l_1+\cdots l_m=k-i}B_{l_1}(z)\cdots B_{l_m}(z)\right)\right)\bigg|_{z=\rho C'(\rho)}.
\end{align*}
From there the result follows, as the second term is exactly the representation of a connected graph of root degree $k-i$ according to the block decomposition, evaluated at $z=\rho$.
\end{proof}

\begin{proof}[Proof of Theorem~\ref{thm:deg}~\eqref{2}]
First we derive a system of functional equations which is satisfied by $C_j'=C_j'(z,\mathbf{v})$, and is refinement of  \[C^\bullet(z) = z \exp (B'(C^\bullet(z))).\]
\begin{lem}\label{Le5}
Let $W_j = W_j(z,\mathbf{v};
C_1',\ldots,C_k',C_\infty')$,
$j \in\{ 1,2,\ldots,k,$ $\infty\}$
defined by
\begin{align*}
W_j &= \sum_{i=0}^{k-j} v_{i+j} C_i'(z,\mathbf{v})
+ v_\infty \left( \sum_{i=k-j+1}^k  C_i'(z,\mathbf{v})
+ C_\infty'(z,\mathbf{v}) \right), \qquad 1\le j\le k,\\
W_\infty &= v_\infty \left( \sum_{i=0}^k  C_i'(z,\mathbf{v})
+ C_\infty'(z,\mathbf{v}) \right).
\end{align*}
Set $\mathbf{W}=(W_1,\ldots,W_k,W_\infty)$.
Then the series
$C_1',\ldots,C_k',C_\infty'$ satisfy the
system of equations
\begin{align*}
C_j'(z,\mathbf{v}) &=
\sum_{\ell_1 +  \cdots j\ell_j= j}
\ \prod_{r = 1}^j \frac{B_r'(z,\mathbf{W})^{\ell_r}}{\ell_r !}, \qquad 1\le j\le k,\\
C_\infty'(z,\mathbf{v}) &=\exp  \Biggl(  \sum_{j=1}^k  B_j'(z,\mathbf{W})+  B_\infty'(z,\mathbf{W}) \Biggr) - 1- \sum_{1\le \ell_1  + \cdots k\ell_k \le k}
\ \prod_{r= 1}^k \frac{B_r'(z,\mathbf{W})^{\ell_r}}{\ell_r !}.
\end{align*}
\end{lem}

\begin{proof}
As already indicated, the proof is a refined version of the functional equation fulfilled by $C^\bullet(z)$, which reflects the decomposition of a derived
connected graph into a finite set of derived $2$-connected
graphs, where every vertex (different from the root) is substituted
by a derived connected graph. The functions $W_j$ serve the purpose of
marking (recursively) the degree of the vertices in the
2-connected blocks which are substituted by other graphs. In the definition of $W_j$, the summation means that we are
substituting a vertex of degree $i$, but since originally the
vertex had degree $j$, we are creating a new vertex of degree
$i+j$, which is marked accordingly by $v_{i+j}$. The analogous holds for $W_\infty$.
\end{proof}
To prove Theorem~\ref{thm:deg}~\eqref{2} we observe
\[
C'(z) = \sum_{0\le j\le \infty} C_j'(z,\mathbf{1}).
\]
Furthermore, since the above system of equations is strongly connected,
all functions $C_j'(z,\mathbf{1})$ have the same radius of
convergence as $C'(z)$, as mentioned in Section \ref{sec:systemequations}. By assumption this radius of convergence
is smaller than the radius of convergence of $B'(z)$.
Hence, by stability if $\mathbf{v}$ is sufficiently close to $\mathbf{1}$
then
the the singularities of $B_j'$ and $C_j'$ do not interfere,
in particular we can apply Theorem \ref{Thsystem1}
and obtain that all functions
$C_j'$ have a square-root singularity.
Finally, let \[C^{(k)}(z,v) = \sum_{n,m}c_{k;n,m} v^m\frac{z^n}{n!}\] be the generating function for the numbers $c_{k;n,m}$ of unrooted connected outerplanar graphs of size $n$ with $m$ nodes of degree $k$. Then $C^{(k)}(z,v)$ satisfies

\[\frac{\partial C^{(k)}(z,v)}{\partial z} = \sum_{j=1}^{k-1} C_j'(z,1,\ldots,1,v,1) + v C'_k(z,1,\ldots,1,v,1) + C'_\infty(z,1,\ldots,1,v,1),\]
and thus $C^{(k)}(z,v)$ has a singular expansion of order $\frac{3}{2}$ around ${v}={1}$.
%
%
%
Furthermore, the central limit theorem for the number of vertices of given degrees with mean $\mu_k n$ and variance $\sigma^2_k n$ follows by the analogue of Theorem \ref{ThcombCLT1} for systems of equations. It immediately follows that $d_k=\mu_k$ as there are exactly $n$ possible ways to root an unrooted object of size $n$ at one of the vertices and thus the probability that a random vertex has degree $k$ is exactly the same as the probability that the root vertex has degree $k$.
Despite that, we could also use formula~\eqref{meanformula} to compute $\mu_k$, we would obtain the same value as for $d_k$ here.
\end{proof}

\subsection{The unlabelled case}\label{subsec:unlab}
We introduce cycle index sums  $$Z_{\widetilde{\mc{B}}'_{j}}(s_1,s_2,\ldots;v_{1,1},v_{1,2},\ldots;\ldots;v_{k,1},v_{k,2},\ldots;v_{\infty,1},v_{\infty,2},\ldots)$$
for the class of pointed blocks, where the pointed vertex has degree $j$ and is not counted, and where the variables $v_{i,j}$ count the cycles of length $j$ of vertices of degree $i$, and $v_{\infty,j}$ counts those vertices of degree greater than $k$. As in Section~\ref{subsec:lab} let $\mathbf{v}=(v_1,\ldots,v_k,v_\infty)$. Denote the corresponding OGFs by $\widetilde{B}'_j(z,\mathbf{v}), j\in \{1,\dots,k,\infty\}$ and let
\begin{align*}
 \widetilde{B}'(z,\mathbf{v}):= \sum_{j=2}^{k} v_j\widetilde{B}'_{j}(z,\mathbf{v})+v_{\infty}\widetilde{B}'_{\infty}(z,\mathbf{v}).
\end{align*}
Note that
\begin{align*}
 Z_{\widetilde{\mc{B}}'}(s_1,s_2,\ldots;1,1,\ldots;\ldots;1,1,\ldots;1,1,\ldots) = Z_{\widetilde{\mc{B}}'}(s_1,s_2,\ldots),
\end{align*}
and thus the singularity $\rho_1(\mathbf{v})$ of $\widetilde{B}'(z,\mathbf{v})$ is the same as that of $\widetilde{B}'(z)$ at $\mathbf{v}=\mathbf{1}$, and $\rho_1(\mathbf{v})$ is the dominant singularity of the system $\mathbf{\widetilde{B}'}(z,\mathbf{v})=(\widetilde{B}'_j(z,\mathbf{v}))_{j\in \{1,\dots,k,\infty\}}$.


We now introduce the multivariate generating functions
\begin{align*}
\widetilde{C'_{j}}(z,\mathbf{v}) &= \sum_{n; n_1,\ldots,n_k,n_{\infty}} \widetilde{c}_{i;n;n_1,\ldots,n_k,n_{\infty}} v_1^{n_1}\cdots v_k^{n_k}v_{\infty}^{n_{\infty}}z^n
\end{align*}
where the coefficient $\widetilde{c}_{i;n;n_1,\ldots,n_k,n_{\infty}}$ denotes the number of elements of size $n$ of $\mc{C}'$, where the pointed vertex has degree $j$ and with $n_i, i\in \{1,\dots,k\}$, vertices of degree $i$ and $n_{\infty}$ vertices of degree greater than $k$. We further set
\[\widetilde{B}'(z,w) = \sum_j \widetilde{B}_j'(z,\mathbf{1})w^j = \sum_{n,j}\widetilde{b}'_{n,j}z^nw^j\]
and
\[\widetilde{C}'(z,w) = \sum_j \widetilde{C}_j'(z,\mathbf{1})w^j = \sum_{n,j}\widetilde{c}'_{n,j}z^nw^j.\]
As we need cycle indices for the block decomposition, we set
\[Z_{\widetilde{\mc{B}}'}(s_1,s_2,\ldots;w) =\sum_j Z_{\widetilde{\mc{B}}_j'}(s_1,s_2,\ldots)w^j.\]
Note that the variable $w$, which counts the degree of the root, is not involved in any permutation cycle.
Then,
\begin{align*}
 \widetilde{C}'(z,w) = \exp \left(\sum_{\ell\geq 1}\frac{1}{\ell}Z_{\widetilde{\mc{B}}'}(z^\ell C'(z^\ell),z^{2\ell}C'(z^{2\ell}),\ldots; w^\ell)\right).
\end{align*}

\begin{thm}\label{thm:deg-unl}
Let $\widetilde{\mathcal{G}}$ be a family of random subcritical graphs
(such that $[z^n]\,\cC(z)> 0$ for $n\ge n_0$)
and $\widetilde{\mathcal{G}}'$ be it's derived family. Further, let $\widetilde{d}_k$ be the limiting probability that the root vertex of a member of $\widetilde{\mathcal{G}}'$ has degree $k$ and let $\widetilde{X}_n^k$ be the random variable that counts the number of vertices of degree $k$ in a randomly chosen member of $\widetilde{\mathcal{G}}$. Then
\begin{enumerate}
 \item\label{un1}
$\widetilde{d}_k=\rho \left( \sum_{i=1}^k \frac{\partial}{\partial z} Z_{\widetilde{\mc{B}}_i'}(z,\rho^2\widetilde{C}'(\rho^2),\rho^3\widetilde{C}'(\rho^3),\ldots)\big|_{z=\rho C'(\rho)}C_{k-i}'(\rho)\right).
$
\item\label{un2} $\widetilde{X}_n^k$ satisfies a central limit theorem with mean $\mathbb{E}\widetilde{X}_n^k=\widetilde{\mu}_k n+O(1)$ and variance $\mathbb{V}{\rm{ar}}\ \widetilde{X}_n^k=\widetilde{\sigma}_k^2 n+O(1)$.
\end{enumerate}
\end{thm}
\begin{rk}
In the unlabelled case, we cannot expect that $\widetilde{\mu}_k$ equals $\widetilde{d}_k$, as vertex-rooting is only possible at fixed points of permutations, and thus in unlabelled graphs $n \widetilde{g_n} \neq \widetilde{g'_n}$.
\end{rk}

\begin{proof}[Proof of Theorem~\ref{thm:deg-unl}~\eqref{un1}] The proof is based on the following lemma:
 \begin{lem}
The generating function $\widetilde{p}(w) = \sum \widetilde{d}_kw^k$ is equal to
\[\rho\frac{\partial}{\partial u}\exp\left(Z_{\widetilde{\mc{B}}'}(u,z^2C'(z^{2}),\ldots;w)+\sum_{\ell\geq 2}\frac{1}{\ell}Z_{\widetilde{\mc{B}}'}(z^\ell C'(z^\ell),z^{2\ell}C'(z^{2\ell}),\ldots; w^\ell)\right),\]
and $\widetilde{p}(1)=1$.  That is, it defines indeed a probability distribution.
\end{lem}
\begin{proof}
We use \cite[Lemma 2.26]{Drmota} with
\begin{align*}
f(z,w)&=zC'(z),\\
H(z,w,u)&=\exp\left(Z_{\widetilde{\mc{B}}'}(u,z^2C'(z^{2}),\ldots;w)+\sum_{\ell\geq 2}\frac{1}{\ell}Z_{\widetilde{\mc{B}}'}(z^\ell C'(z^\ell),\dots; w^\ell)\right),
\end{align*}
and the same line of reasoning as in \cite[Theorem 9.17]{Drmota} proves the first part of the lemma. For $\widetilde{p}(1)=H_u (\rho,1,\rho C'(\rho))$ we obtain
\[\rho C'(\rho) \frac{\partial}{\partial u}Z_{\widetilde{\mc{B}}'}(u,\rho^2 C'(\rho^2),\ldots)\bigg|_{u=\rho C'(\rho)} =1\] by the implicit function representation of $C'(z)$.
\end{proof}
Now we can determine $\widetilde{d}_k=[w^k]\widetilde{p}(w)$, which is equal to
\begin{align*}
\rho\ [w^k]\left[\exp\left(\sum_{\ell \geq 1}\frac1{\ell}Z_{\widetilde{\mc{B}}'}(\rho^\ell\widetilde{C}'(\rho^\ell),\ldots;w^\ell)\right) \left( \frac{\partial}{\partial u} Z_{\widetilde{\mc{B}}'}(u,\rho^2\widetilde{C}'(\rho^2),\ldots;w)\right)_{u=\rho C'(\rho)}\right].
\end{align*}
Recalling $Z_{\widetilde{\mc{B}}'}= \sum_{l \geq 1} Z_{\widetilde{\mc{B}}_{\ell}'}$, we obtain that the previous expression can be written
\begin{align*}
\rho\ \sum_{i=1}^k \diff{u} Z_{\widetilde{\mc{B}}_i'}(u,\rho^2\widetilde{C}'(\rho^2),\ldots)\big|_{u=\rho C'(\rho)} \cdot [w^{k-i}]\exp\left(\sum_{\ell \geq 1}\frac{1}{\ell}Z_{\widetilde{\mc{B}}'}(\rho^\ell\widetilde{C}'(\rho^\ell),\ldots;w^\ell)\right).
\end{align*}
Observe that the second term translates into $\widetilde{C}'_{k-i}(\rho)$, as the following equality holds: $$\exp\left(\sum_{\ell \geq 1}\frac{1}{\ell}Z_{\widetilde{\mc{B}}'}(\rho^\ell\widetilde{C}'(\rho^\ell),\ldots;w^\ell)\right)=\widetilde{C}'(\rho,w)= \sum_j \widetilde{C}_j'(\rho,\mathbf{1})w^j= \sum_j \widetilde{C}_j'(\rho)w^j.$$
\end{proof}
\begin{proof}[Proof of Theorem~\ref{thm:deg-unl}~\eqref{un2}]
As in the labelled case, we observe that the functions $\widetilde{C'_{j}}(z,\mathbf{v})$ satisfy a system of equations, using a refinement of the block decomposition. In the following we will denote by $\sum_{i=r}^{k,\infty}F_i=\sum_{i=r}^kF_i+F_\infty$.

\begin{lem}\label{unlsys}
For each $j=1,2,\ldots,k,\infty$, let $W_j$ be defined by
\begin{align*}
&W_j(z,\mathbf{v}) = \sum_{i=0}^{k-\ell} v_{j+i} \tcsef{j} +v_\infty \left(\sum_{i=k-j+1}^{k,\infty} \tcsef{i}\right),\\
&W_\infty(z,\mathbf{v}) =v_\infty \left(\sum_{i=0}^{k,\infty}\tcsef{i}\right).
\end{align*}
Set $W_{j,i} = W_j(z^i,v_1^i,\ldots,v_k^i,v_{\infty}^i)$ and
$$\mathbf{W}^{(l)}= (W_{1,l},W_{1,2l},\ldots;\ldots; W_{k,l},W_{k,2l},\ldots;W_{\infty,l},W_{\infty,2l},\ldots).$$
We denote by $S_n(s_1,s_2,\ldots)$ the cycle index of the symmetric group on $n$ elements. Furthermore $S_n[Z_\mc{B}]$ denotes the substitution $s_l\gets Z_{\mc{B}}(s_l,s_{2l},\ldots;\mathbf{W}^{(l)}), l \geq 1$ in $S_n(s_1,s_2,\ldots)$.

Then the series $\widetilde{C'_{1}},\ldots,\widetilde{C'_{k}},\,\widetilde{C'_{\infty}}$ satisfy the system of equations
\begin{align*}
\tcsef{j}
&= \sum_{l_1+2l_2+\cdots+jl_j=j} \ \prod_{r=1}^{j} S_{l_r}\left[Z_{\mc{B}'_r}\right]_{(s_i=z^i)},\quad 1\leq j \leq k\\
\tcsef{\infty}
&=\exp\left(\sum_{l \geq 1} \frac{1}{l}\left(\sum_{r=1}^{k} Z_{\mc{B}'_r}\left(z^l,z^{2l},\ldots; \mathbf{W}^{(l)}\right)+Z_{\mc{B}'_{\infty}}\left(z^l,z^{2l},\ldots; \mathbf{W}^{(l)}\right)\right) \right)-\\
& \sum_{1 \leq l_1+\cdots+kl_k\leq k}\ \prod_{r=1}^{j} S_{l_r}\left[Z_{\mc{B}'_r}\right]_{(s_i=z^i)}.
\end{align*}
%
%
\end{lem}
\begin{rk}
The functions $W_{j,i}$ and thus the whole system can also be considered in terms of cycle index sums, using the cycle index sum $$Z_{\widetilde{C}_j'}(s_1,s_2,\ldots;v_{1,1},v_{1,2},\ldots;\ldots;v_{k,1},v_{k,2},\ldots;v_{\infty,1},v_{\infty,2},\ldots)$$ for rooted connected graphs. The root vertices in $W_{j,i}$ are fixed and thus have cycle length $1$. We will need this terminology in the proof of Lemma \ref{lem:unrooting}.
\end{rk}
\begin{proof}
As in the labelled case, we refine the recursive decomposition of graphs into it's $2$-connected components. The functions $W_{j,i}$ plays the analogous role as in the labelled case, except that we need a second index for representing the cycles of different length appearing in the cycle indices.
This directly leads to $\tcsef{j}$ for $j=1,\ldots k$. For $\tcsef{\infty}$ we obtain
\begin{align*}
&\sum_{(l_1,l_2,\ldots,l_k,l_\infty)}\prod_{r=1}^{k} S_{l_r}\left[\Zb{r}\left(z,z^2,\ldots;\mathbf{W}^{(1)}\right)\right] S_{l_{\infty}}\left[ \Zb{\infty}\left(z,z^2,\ldots;\mathbf{W}^{(1)}\right)\right]-\\
&\ \sum_{l_1+\ldots  +k l_k\leq k} \prod_{r=1}^{k} S_{l_r}\left[\Zb{r}\left(z,z^2,\ldots;\mathbf{W}^{(1)}\right)\right],\\
\end{align*}
where the first sum rewrites into the exponential term appearing in $\tcsef{\infty}$.
\end{proof}

It is easily checked that the system is strongly connected as every $\tcsef{j}$ depends on $\tcsef{\infty}$ for $j\in \{1,\ldots, k\}$ and $\tcsef{\infty}$ depends on all the values $j\in\{ 1,\ldots,k\}$. Obviously,
\begin{align}\label{eq:veq1}
 \widetilde{C}'(z)=\left[\sum_{j=1}^{k}v_j\tcsef{j}+v_{\infty}\tcsef{\infty}\right]_{\mathbf{v}=\mathbf{1}}.
\end{align}
Define
\begin{align*}
\widetilde{C}'(z,\mathbf{v}):=\sum_{j=1}^{k}v_j\tcsef{j}+v_{\infty}\tcsef{\infty}.
\end{align*}
Then the generating function of derived connected graphs of $\widetilde{\mathcal{G}}$, where the variable $v$ counts the nodes of degree $k$, is given by $\widetilde{C}'^{(k)}(z,v)=\widetilde{C}'(z,\mathbf{v}_k)$, where $\mathbf{v}_k=(1,1,\ldots,1,v,1)$.

\begin{lem}
$\widetilde{C}'^{(k)}(z,v)$ has a square-root singular expansion around its singularity $\rho_2(v)$ in a neighbourhood of $v=1$.
\end{lem}

\begin{proof}
By Equation \eqref{eq:veq1} the singularity of the system $\mathbf{\widetilde{C}'}(z,\mathbf{v})$ at $\mathbf{v}=\mathbf{1}$ is $\rho_2$, the same as that of $\widetilde{C}'(z)$. As the system is strongly connected, every $\widetilde{C}'_j(z,\mathbf{1})$ has radius of convergence $\rho_2$. Since $\widetilde{C}'^{(k)}(z,1)$ is a linear combination of these functions, it has the same singularity, which fulfills $\rho_2\widetilde{C}'(\rho_2)<\rho_1$ due to the subcriticality assumption. By the stability property of subcriticality it follows that $(\widetilde{C}'^{(k)}(z,v),\widetilde{B}'^{(k)}(z,v))$ is subcritical near $1$, and hence we obtain a square-root singular expansion.  
\end{proof}

Consider the cycles index sums
$$Z_{\widetilde{\mc{C}}}(s_1,s_2,\ldots;v_{1,1},v_{1,2},\ldots;\ldots;v_{k,1},v_{k,2},\ldots;v_{\infty,1},v_{\infty,2},\ldots)$$
for (unpointed) connected graphs of $\widetilde{\mathcal{G}}$.
By taking $s_1 = s$, $s_i = z^i$ for $i \geq 2$, $v_{j,i} = 1$ for $1\le j <k, i\ge 1$ and $v_{k,i} = v^i$ for $i\ge 1$ we obtain its corresponding OGF  $\widetilde{C}^{(k)}(z,v)=Z_{\widetilde{\mc{C}}}(z,z^2,\ldots;1,1,\ldots;\ldots;v,v^2,\ldots;1,1,\ldots)$, where the variable $v$ counts the nodes of degree $k$.

\begin{lem}\label{lem:unrooting}
 $\widetilde{C}^{(k)}(z,v)$ has a singular expansion of order $\frac{3}{2}$ around its singularity $\rho_2(v)$ in a neighbourhood of $v=1$.
\end{lem}

\begin{proof}
We have to express the system of equations in Lemma \ref{unlsys} in terms of cycle index sums and analyze the trivariate generating functions \[\widetilde{C}^{(k)}(s,z,v)=Z_{\widetilde{\mc{C}}}(s,z^2,z^3,\ldots;1,1,\ldots;\ldots;v,v^2,\ldots;1,1,\ldots).\]
Obviously, $\widetilde{C}^{(k)}(z,z,v) = \widetilde{C}^{(k)}(z,v)$. Analogously we define $\widetilde{C}'^{(k)}(s,z,v)$. We obtain
\begin{align*}
\widetilde{C}^{(k)}(s,z,v) =\widetilde{C}^{(k)}(0,z,v) + \int_0^z\widetilde{C}'^{(k)}(s,z,v) \mathrm{d} s,
\end{align*}
by the same arguments as in Section \ref{sec:sub_unlabelled}. Due to stability, we obtain a square-root singular expansion for $\widetilde{C}'^{(k)}(s,z,v)$ as before, with a singular term of the form $(1-s/\bar{\rho}(z,v))^{1/2}$. Integration and subcriticality conditions lead to a singular expansion of the form
\[\widetilde{C}^{(k)}(s,z,v) = a(s,z,v)+b(s,z,v)\left(1-\frac{s}{\bar{\rho}(z,v)}\right)^{3/2}.\] At $s=z$ we can represent the singular part as \[\left(1-\frac{z}{\bar{\rho}(z,v)}\right)^{3/2} = \kappa(z,v)\left(1-\frac{z}{\tilde{\rho}_2(v)}\right)^{3/2},\] with an analytic factor $\kappa(z,v)$ and some analytic function $\tilde{\rho}(v)$. Since the singular manifold of $\widetilde{C}'^{(k)}(s,z,v)$ is the same as that of $\widetilde{C}^{(k)}(z,v)$, that is $z=\bar{\rho}(z,v)$ if and only if $z=\rho_2(v)$, it follows that $\tilde{\rho}(v) = \rho_2(v)$.
\end{proof}
Now, with the singular expansion in $\rho(v)$ given, we can immediately deduce a central limit law by the analogue to Theorem \ref{ThcombCLT1} for systems of equations.
\end{proof}

To calculate $\widetilde{\mu}_k$ we can use Expression~\eqref{meanformula}. The derivatives give
\begin{align*}
\mathbf{b^T} \mathbf{F}_v(\mathbf{C}(\rho,1); \rho,1)&=\left(E\cdot
\sum_{j=1}^{k,\infty} \sum_{i=1}^k (B'_{j})_{v_i}(z,\mathbf{W})
y_{k-i}\right)\bigg|_{(\mathbf{C}(\rho,1);\rho,1)},\\
 \mathbf{b^T} \mathbf{F}_z(\mathbf{C}(\rho,1);\rho,1)&=\left(E\cdot
\sum_{j=1}^{k,\infty}(B'_{j})_{z}(z,\mathbf{W})\right)\bigg|_{(\mathbf{C}(\rho,1);\rho,1)},
\end{align*}
where $y_{k-i}$ denotes the $(k-i)$-th coordinate of $\vy$
satisfying $\vy=\mathbf{F}(\vy;z,\vv)$, $(B'_{j})_{v_i}$ is the derivative of $B'_{j}$ with respect to $v_i$ and $E$ is the exponential term appearing in $\widetilde{C}'_\infty$. As the formula includes all partial derivatives of $(\widetilde{B}'_j)$, we see that the calculation of $\widetilde{\mu}_k$ will be very involved and it is (in general) not equal to $\widetilde{d}_k$.

\section{Computing the growth rate}\label{sect: constant}
In this section we show how to compute the exponential growth of
unlabelled $2$-connected, connected and general SP-graphs,
respectively. The main tool is the decomposition grammar and the systems of functional equations developed in Section~\ref{sec_sp_unl}.

We will first review the growth constants for subcritical classes of graphs in Section~\ref{sect:constants}.
Next we will approximate the exponential growth for the number
of unlabelled SP-graphs in two steps: the first step for $2$-connected unlabelled SP-graphs (Section~\ref{sect:2-conn-constant})
and the second step for unlabelled connected SP-graphs (Section~\ref{sect:conn-constant}).
To this end, we start with a functional system of the form $\textbf{y}=\textbf{F}_N(\textbf{y};z)$,
which is a truncated version (combined with iteration) of the original functional system $\textbf{F}(\textbf{y};z)$ determining $\textbf{y}$  and provides an \emph{approximation} of $\textbf{y}$. The solution of this system, together with the additional restriction given by a singular equation (given by the
determinant of the Jacobian of the previous  system) gives the desired growth constant.

The error introduced in the computations becomes smaller as soon as we take a better approximation for $\textbf{y}$
(in other words, more terms on the Taylor series of its components). We refer the reader to~\cite{BoFuKaVi07b} and~\cite{PiSaSo08},
in which the authors study a similar iterative scheme.

\subsection{Explicit growth constants}\label{sect:constants}

The computation of the growth constant for the classes under study
is based on the subcritical condition. In labelled case, the singularity $\rho$ is given by the equation
$\rho=\tau \exp\left(-\cB'(\tau)\right)$, where $\tau$ is the
smallest solution of the equation $y\cB''(y)=1$.
In case of unlabelled case, the growth constants obtained earlier in the literature are based on
\emph{explicit} expressions for $Z_{\cB'}(s_1,s_2,\dots)$, e.g. the classes of cacti graphs and outerplanar graphs.
Here we will show how to derive the growth constants of SP graphs, even when there is not an explicit expression for $Z_{\cB'}(s_1,s_2,\dots)$.

We have shown in Theorem~\ref{theo:ex} that the asymptotic number of cacti, outerplanar and SP graphs is of the form $c
\ \! n^{-5/2}\ \! \rho^{-n}\ \! n!\ \! (1+o(1))$ in the labelled case
and is of the form $c\ \! n^{-5/2}\ \! \rho^{-n}\ \! (1+o(1))$ in the unlabelled case.
In both cases (either labelled or unlabelled) the subexponential term $n^{-5/2}$ suggests that
all these classes have an \emph{arborescent} structure. In Table~\ref{table:constant-growth} the exponential growth constant $\rho^{-1}$ for connected
(and general) acyclic, cacti, outerplanar and SP-graphs are shown.
\begin{table}[htb]
\begin{center}
\begin{tabular}{|c||c|c|}
  \hline
  Family            &     Labelled          & Unlabelled
  \\\hline\hline
Acyclic               & $2.71828$ & $2.95577$\\\hline
  Cacti             & $4.18865$          & $4.50144$                        \\\hline
  Outerplanar       & $7.32708$        & $7.50360$                    \\\hline
  Series-Parallel   & $9.07359$           & $9.38527$                   \\\hline
\end{tabular}\bigskip
\caption{Exponential growth for distinct subcritical classes. All
constants are referred to connected and general classes.
}\label{table:constant-growth}
\end{center}
\end{table}
\subsection{Parameters for 2-connected SP-graphs.}\label{sect:2-conn-constant}
First we
compute the growth constant using the equations for networks (and
related classes). Later, the resulting singularity is transferred
to the counting series of $2$-connected SP-graphs
$Z_{\cB}(z,z^2,z^3,\dots)$.

Denote by  $Z_{\mathcal{S}}(s_1,s_2,\dots)$,
$Z_{\mathcal{P}}(s_1,s_2,\dots)$ and
$Z_{\mathcal{D}}(s_1,s_2,\dots)$ the cycle index sums associated
to series, parallel and general networks. Let
$S(z)=Z_{\mathcal{S}}(z,z^2,\dots)$,
$P(z)=Z_{\mathcal{P}}(z,z^2,\dots)$,
$D(z)=Z_{\mathcal{D}}(z,z^2,\dots)$. As it is shown in
Section~\ref{sec:ex}, the
system of equations which defines series, parallel and general
networks is
\begin{equation}\label{eq:networks-system}
\left\{
\begin{array}{rcl}
 D(z)&=&1+S(z)+P(z),\\
S(z)&=&zD(z)\left(1+P(z)\right),\\
P(z)&=&2\exp\left(\sum_{i\geq 1}\frac{1}{i}S^{(i)}(z)\right)-S(z)-2.
\end{array}
\right.
\end{equation}
Let $U(z)=\exp\left(\sum_{i\geq 2}\frac{1}{i}S^{(i)}(z)\right)$. This function is analytic at the singularity of $S(z)$.
We can isolate
$S(z)$ from System~\eqref{eq:networks-system} and obtain the implicit
relation
$$S(z)=z\left(2\exp(S(z))U(z)-1\right)\left(2\exp(S(z))U(z)-1-S(z)\right).$$
Observe that $U(z)$ depends on $S(z)$. In order to get an approximation of the smallest singularity of $S(z)$,
we start looking for an approximation of $S(z)$. This can be done
in the following way: let $S_0(z)=0$ and $U_0(z)=1$. Let $N$ be a
positive integer. For a function $f(z)$, analytic at the origin,
define $\mathrm{\texttt{pol}}_N\left\{f(z)\right\}$ as the Taylor
polynomial of degree $N$ of $f(z)$. Define $S_{k+1}(z)$ and
$U_{k+1}(z)$ recursively in the following way:
%
%
\begin{equation*}\label{eq:iteration}
\left\{
\begin{array}{rcl}
S_{k+1}(z)&=&\mathrm{\texttt{pol}}_N\left\{z\left(2\exp(S_{k}(z))U_{k}(z)-1\right)
\left(2\exp(S_{k}(z))U_{k}(z)-1-S_{k}(z)\right)\right\}, \\
U_{k+1}(z)&=&\mathrm{\texttt{pol}}_N\left\{\exp\left(\sum_{i=2}^{N}\frac{1}{i}S_{k+1}^{(i)}(z)\right)
\right\}.
\end{array}
\right.
\end{equation*}
We stop when $S_{k-1}(z)=S_{k}(z)$ (equivalently, when
$S_{k}(z)=\mathrm{\texttt{pol}}_N\{S(z)\}$). Denote then
$u(z)=\exp\left(\sum_{i=2}^{N}\frac{1}{i}S_{k}^{(i)}(z)\right)$,
and consider the solution $s(z)$ of the equation
$$s(z)=z\left(2\exp(s(z))u(z)-1\right)\left(2\exp(s(z))u(z)-1-s(z)\right).$$
Taking $u(z)$ instead of $U(z)$ introduces an error, which can be controlled in that the larger $N$ is, the smaller the error. Under these assumptions, $s(z)$ is defined by an equation
of the form $s=H(s;z)$, with
$H(y;z)=z(2\exp(y) u(z)-1)(2\exp(y)
u(z)-1-y)$. Consequently, we find its smallest singularity by
solving the characteristic system $y=H(y;z),\, 1=H_{y}(y;z)$. In Table~\ref{table:rho-biconnected} the values obtained for several choices are shown. In particular, one can observe that the accuracy in the computation are improved by increasing the order of truncation.
\begin{table}[htb]
\begin{center}
\begin{tabular}{|r|l|}
  \hline
   $N$            &     Singular point ($\rho$)   \\\hline\hline
$5$& $\textbf{0.124}21863192426192376$  \\\hline
$10$& $\textbf{0.124199}19715484630978$          \\\hline
$20$& $\textbf{0.12419909378}526277564$        \\\hline
$50$& $\textbf{0.12419990937841528588}$           \\\hline
%
\end{tabular}\bigskip
\caption{Values for $N$ and the corresponding value for $\rho$. 
}\label{table:rho-biconnected}
\end{center}
\end{table}

These computations give the truncated value of $z$, $\rho_1\approx0.12420$. This constant is
slightly smaller than the one which is obtained in the labelled case
(whose value is approximately $0.12800$. See~\cite{BoGiKaNo07}). This singularity is the same for the parallel family and for general networks.
Using the same ideas one finds that the smallest singularity for exchanging poles networks is strictly bigger than $\rho_1$.
Consequently, applying the transfer theorems of singularity analysis we conclude that the number $\widetilde{b_n}$ of $2$-connected unlabelled
series-parallel graphs on $n$ vertices is
$$\widetilde{b_n}=\widetilde{ b}\ \! n^{-3/2}\ \! \gamma_1^{n} \ \!(1+o(1)),$$
where $\gamma_1=\rho_1^{-1}\approx 8.05159$ and $\widetilde{b}$ is a constant.
\subsection{Parameters for connected SP-graphs.}\label{sect:conn-constant}
In order to approximate the growth constant for connected SP-graphs,
we need to refine the analysis over the cycle index sum for pointed 2-connected unlabelled SP-graphs, which is
defined in terms of simpler pointed classes using the dissymmetry
theorem for tree-decomposable structures developed in~\cite{ChFuKaSh07}. In particular, we have that $Z_{\cB'}=\ell'+Z$,
where
$$Z=Z_\mathcal{R}+Z_{\mathcal{M}}-Z_{\mathcal{R}-\mathcal{M}}.$$
Here $\ell'$ refers to a pointed link-graph
(with cycle index sum $s_1$), and $Z_\mathcal{R}$,
$Z_\mathcal{M}$, $Z_{\mathcal{R}-\mathcal{M}}$ are the cycle index sums associated to the classes
of pointed 2-connected graphs with a pointed $\mathcal{R}$-brick, a
pointed $\mathcal{M}$-brick and an a pointed edge
$\mathcal{R}-\mathcal{M}$ in the associated
$\mathcal{R}\mathcal{M}\mathcal{T}$-tree (see~\cite{ChFuKaSh07} for
proper definitions). Each one of these series can be
written in terms of the cycle index sum associated to series
networks, general networks and networks which remains invariant when
a change of the pole is applied. Denoting these series by
$Z_\mathcal{S},Z_\mathcal{D}$ and $\overline{Z}_\mathcal{D}$,
respectively, we get
\begin{equation}\label{eq:sys1}
\left\{
\begin{array}{lll}
  Z_{\mathcal{R}}&=&\frac{1}{2}\left(s_1^2\left(Z_\mathcal{D}-Z_\mathcal{S}\right)^2Z_\mathcal{D}+s_2\left(Z_\mathcal{D}^{(2)}-Z_\mathcal{S}^{(2)}\right)\overline{Z}_\mathcal{D}\right) \\
  Z_{\mathcal{M}}&=& s_1\left(2\exp\left(\sum_{i>0}\frac{1}{i}Z_{\mathcal{S}}^{(i)}\right)-2-2Z_{\mathcal{S}}-\frac{1}{2}\left(Z_{\mathcal{S}}^2+Z_{\mathcal{S}}^{(2)}\right)\right) \\
  Z_{\mathcal{R}-\mathcal{M}}&=& s_1 Z_{\mathcal{S}}\left(2\exp\left(\sum_{i>0}\frac{1}{i}Z_{\mathcal{S}}^{(i)}\right)-2-Z_{\mathcal{S}}\right)
\end{array}
\right.
\end{equation}
Recall that in the previous equation
$Z_\mathcal{\star}^{(i)}=Z_\mathcal{\star}(s_{i},s_{2i},\dots)$. We denote by $F(z)$ the solution of the equation
$F(z)=z\ \!\overline{\exp}(Z_{\cB'}(F(z),F(z^2),\dots))$ ($\overline{\exp}$ denotes the Polya operator for sets in the unlabelled framework). This
equation can be written in the following way: write
$J(z)=Z_{\cB'}(F(z),F(z^2),\dots)$ and $A(z)=\exp\left(\sum_{i\geq
2}\frac{1}{i}J\left(z^{i}\right)\right)$. Consequently,
$F(z)=z\exp\left(J(z)\right)A(z)$. We define
 $\bS(z)=Z_{\mathcal{S}}(F(z),F(z^2),\dots)$ and
$\overline{\bS}(z)=\overline{Z}_{\cS}(F(z),F(z^2),\dots)$. Define
also the following functions: $E(z)=\exp\left(\sum_{i\geq
1}\frac{1}{2i}\bS\left(z^{2i}\right)\right)$,
$M(z)=\exp\left(\sum_{i\geq
1}\frac{1}{2i+1}\bSb\left(z^{2i+1}\right)\right)$, and
$G(z)=\exp\left(\sum_{i\geq
2}\frac{1}{i}\bS\left(z^{i}\right)\right)$. Observe that all $A(z)$,
$E(z)$, $G(z)$ and $M(z)$ are analytic functions at the singularity
of $F(z)$ (which is the same singularity as the one of $J(z)$).
With these definitions and using System~\eqref{eq:sys1} we get the
following system of equations:
\begin{eqnarray*}
J(z)&=&z\exp(J(z))A(z)+\\\nonumber
    &&\left(z\exp(J(z))A(z)\left(2\exp(\bS(z))G(z)-1-\bS(z)\right)\right)^2(\exp(\bS(z))G(z)-1/2)+\\
    &&F(z^2)\left(2\exp(\bS(z^2))G(z^2)-1-\bS\left(z^2\right)\right)\left(E(z)M(z)\exp(\bSb(z))-1/2\right)+\\\nonumber
    &&z\exp(J(z))A(z)\left(2\left(\exp(\bS(z))G(z)-1-\bS(z)\right)-\left(\bS(z)^2+\bS\left(z^2\right)\right)/2\right)-\\\nonumber
    &&z\exp(J(z))A(z)\bS(z)\left(2\exp(\bS(z))G(z)-2-\bS(z)\right)\\\nonumber
\bS(z)&=&\left(2\exp(\bS(z))G(z)-\bS(z)-1\right)z\exp(J(z))A(z)\left(2\exp(\bS(z))G(z)-1\right)\\\nonumber
\bSb(z)&=&\left(2\exp(\bS\left(z^2\right))G\left(z^2\right)-1\right)z\exp(J(z))A(z)+\\
       &&\left(2\exp(\bS\left(z^2\right))G\left(z^2\right)-1\right)F(z^2)\left(2E(z)M(z)\exp(\bSb(z))-1-\bSb(z)\right).\nonumber
\end{eqnarray*}
By a fixed-point argument we are able to compute from this system
the first terms in the Taylor development of the series which appear
in the previous equations. The procedure is the same way as in Section~\ref{sect:2-conn-constant} (taking, for instance, the initial
conditions $A(z)=E(z)=G(z)=M(z)=1$ and $F(z)=\bS(z)=\bSb(z)=0$). Later, we consider the simplified
system in which we approximate all analytic functions by their Taylor series, up to a prescribed index. Consequently, the reduced system can be written in the compact form
$(J,\bS,\bSb)=\textbf{F}(J,\bS,\bSb;z)$. In order to find the critical points, we consider the associated singular system of equations, which is obtained from the system $(J,\bS,\bSb)=\textbf{F}(J,\bS,\bSb;z)$
by considering the determinant of its Jacobian.

Solving this system using a symbolic manipulator (for instance, \texttt{Maple}), we get that the singular value of $z$ is $\rho_2\approx 0.10655$ for a truncation to order $N=50$. Hence, the numbers $\widetilde{c_n}$ and $\widetilde{g_n}$ of unlabelled
connected and general SP-graphs on $n$ vertices are
$$\widetilde{c_n}=\widetilde{ c}\ \! n^{-5/2}\ \!\gamma_2^n\ \!(1+o(1)),\quad \widetilde{g_n}= \widetilde{g}\ \! n^{-5/2}\ \!\gamma_2^n\ \!(1+o(1)), $$
where $\gamma_2\approx 9.38527$ and $\widetilde{c},\,\widetilde{g}$ are constants.
\\
\begin{paragraph}{\textbf{Acknowledgements.}} Bilyana Shoilekova, Stefan Vigerske, Philippe Flajolet and Marc Noy are greatly thanked for inspiring discussions and useful comments. We also thank the anonymous referees for their useful suggestions.
\end{paragraph}

\bibliography{subcritical}
\bibliographystyle{abbrv}

\end{document}